\documentclass[reqno]{amsart}
\usepackage[utf8]{inputenc}
\usepackage{amsmath}
\usepackage[shortlabels]{enumitem}
\usepackage {amssymb}
\usepackage{amsfonts}
\usepackage{amsthm}
\usepackage{color}
\usepackage{natbib}
\usepackage{graphicx}
\usepackage{mathtools}
\usepackage{MnSymbol}
\usepackage{hyperref}
\usepackage{algorithm}

\usepackage{caption}
\usepackage{subcaption}
\usepackage{float}
\usepackage[noend]{algpseudocode}
\hypersetup{colorlinks,citecolor=blue,urlcolor=blue,linkcolor=blue}
\theoremstyle{plain}

\newtheorem{definition}{Definition}
\newtheorem{corollary}{Corollary}

\newtheorem{theorem}{Theorem} 
\newtheorem{lemma}{Lemma}
\newtheorem{remark}{Remark}

\newtheorem{proposition} {Proposition} 
\newtheorem*{proposition-non} {Proposition} 
\newtheorem*{theorem-non}{Theorem}

\newtheorem{assumption}{Assumption}
\newtheorem{assumptions}{Assumptions}

\numberwithin{theorem}{section} 
\numberwithin{proposition}{section} 
\numberwithin{corollary}{section} 
\numberwithin{assumptions}{section} 
\numberwithin{assumption}{section} 
\numberwithin{lemma}{section}
\numberwithin{definition}{section}
\numberwithin{remark}{section}
\numberwithin{algorithm}{section}

\newcommand{\Var}[2][]{\operatorname{Var}_{#1}\left[#2\right]}

\newcommand{\abs}[1]{\left\lvert#1\right\rvert}
\newcommand{\Norm}[1]{\left\lVert#1\right\rVert}

\DeclareMathOperator*{\argmin}{argmin}   

 \newcommand\omicron{o}

\DeclarePairedDelimiter{\diagfences}{(}{)}
\newcommand{\diag}{\operatorname{diag}\diagfences}

\date{}

\begin{document}

\author{Panagiotis Lolas $^{*1}$}
\thanks{$^*$panagd@stanford.edu}

\author{Lexing Ying$ \; ^{\dagger 1}$}
\thanks{$^{\dagger}$lexing@stanford.edu}
\thanks{$^{1}$ Department of Mathematics, Stanford University}

\title{Shrinkage Estimation of Functions of Large Noisy Symmetric Matrices}

\maketitle

\begin{abstract}
  We study the problem of estimating functions of a large symmetric matrix $A$ when we only have
  access to a noisy estimate $\hat{A}_n=A_n+\sigma Z_n/\sqrt{n}.$ We are interested
  in the case that $Z_n$ is a Wigner ensemble and suggest an algorithm based on nonlinear shrinkage of
  the eigenvalues of $\hat{A}_n.$ As an intermediate step we explain how recovery of the spectrum of
  $A_n$ is possible using only the spectrum of $\hat{A}_n$. Our algorithm has important applications,
  for example, in solving high-dimensional noisy systems of equations or symmetric matrix
  denoising. Throughout our analysis we rely on tools from random matrix theory.
\end{abstract}

\section{Introduction}\label{sec:intro}

\subsection{Problem and Assumptions}\label{subsec:assum}
Let $A_n\in\mathbb{R}^{n\times n}$ be a real symmetric matrix (deterministic or random), which is
unknown. Instead, we have access to a noisy estimate $\hat{A}_n = A_n + \sigma n^{-1/2}Z_n.$ We will
often omit the subscript $n$ in our notation. We will denote by $\lambda_1\geq \lambda_2\geq \cdots
\geq \lambda_n$ the eigenvalues of $A_n,$ $w_1,\cdots,w_n$ the corresponding eigenvectors and the
empirical spectral distribution of $A_n$ by $\mu_n.$ The latter is the measure $\mu_n =
n^{-1}\sum_{k=1}^n\delta_{\lambda_k}.$ Similarly we are going to denote by
$\hat{\lambda}_1\geq\cdots\geq\hat{\lambda}_n$ the eigenvalues of $\hat{A}_n$ and
$\hat{w}_1,\cdots,\hat{w}_n$ the corresponding eigenvectors.

\begin{assumptions}\label{assume:1}
  We assume that $A_n$ and $Z_n$ satisfy the following assumptions.
  \begin{enumerate}
  \item The dimension $n$ of the matrix $A_n$ goes to infinity. 
  \item The spectral distribution of the eigenvalues of $A_n$ converges weakly almost surely to a
    deterministic probability measure $H.$
  \item The measure $H$ is supported on a compact interval contained in $\mathbb{R}$ and eventually
    all of the eigenvalues of $A$ lie in a compact subset $[h_1,h_2]$ of $\mathbb{R}.$
  \item The matrix $Z_n$ is real symmetric and independent of $A_n.$ The matrix $Z_n$ is a submatrix of an
    infinite matrix $(Z_{ij})_{1\leq i,j\leq n}$ whose upper half has i.i.d. entries with mean 0,
    variance $1$ and finite fourth moments. 
\end{enumerate}

\end{assumptions}

We are interested in estimating $h(A_n),$ where $h$ is a continuous function defined on an open set that contains $[h_1,h_2]$. Special cases include, for example, $h(x)=x$ (which is the problem of denoising $\hat{A}_n$), or $h(x)=x^{-1},$ which is interesting for solving noisy linear systems of equations. Other interesting choices might include $h(x)=\sqrt{x}$ (estimating the square root of a positive semi-definite matrix), or $h(x)=x/(x^2+\lambda^2)$ (for estimating the regularized inverse of a symmetric matrix).

\subsection{Our Contributions}

The main contributions of our paper are listed below:

\begin{enumerate}
    \item We derive (in closed form) the optimal nonlinear shrinkage for estimating $h(A_n)$ in Frobenius loss. 
    \item We suggest a practical algorithm that asymptotically estimates the optimal nonlinear shrinkage for any choice of function $h.$
    \item We study the problem of recovering the limiting spectral distribution $H$ of the matrix $A_n.$ We consider the cases of known and unknown noise level $\sigma^2.$ Recovering the measure $H$ is important for the implementation of our algorithm.    \item We show how our results can be used to derive the optimal  shrinkage function with alternative choices of losses.
    \item We study asymptotic expansions of the optimal shrinkers when $\sigma\rightarrow 0$ and $\sigma\rightarrow\infty.$

\end{enumerate}

\subsection{Related Work}

Shrinkage methods have been used in statistics in different settings with great success. In \cite{james_stein} the authors showed how estimation of the mean of a Gaussian distribution in more than 2 dimensions can be improved significantly by shrinkage of the sample estimates. For the purpose of covariance matrix estimation, linear shrinkage methods were used in \cite{ledoit_linear} to suggest a well-conditioned estimator of a high-dimensional covariance matrix. Using tools from random matrix theory, in \cite{ledoit2012nonlinear} the authors showed how nonlinear shrinkage methods can be used to greatly improve estimation and a nonparametric procedure that achieves greater speed and numerical stability was suggested in \cite{ledoit2020analytical}. For the case of spiked models, \cite{donoho_gavish_johnstone} used nonlinear shrinkage to estimate the population covariance matrix and derived the optimal shrinker for 26 losses, for most of them in closed form. For regularization of linear discriminant analysis, general nonlinear eigenvalue shrinkage was used in \cite{lolas2020regularization} to improve the classification accuracy when the feature dimensionality is comparable to the number of samples and sharp classification error asymptotics for any shrinkage function were derived. 

For the case of a deformed Wigner model as the one we consider here, \cite{donoho_gavish_smd} showed how eigenvalue shrinkage can be used for symmetric matrix denoising in the case that $A$ is low-rank. For the problem considered here, $h(x)=x$ was studied by \cite{bouchaudrotational}, where the authors derived the optimal nonlinear shrinkage in closed form using replica symmetry. In that case the authors showed that, given $\sigma,$ the optimal shrinker depends on $H$ only through the Stieltjes transform of the limiting spectral distribution of $\hat{A}_n.$ This phenomenon makes the optimal shrinkage function easy to estimate (for example, with a similar nonparametric procedure as in \cite{ledoit2020analytical}).

The problem of numerical computation of the free-convolution of two probability measures has been
studied in \cite{rao_edelman}, \cite{rao_numerical}. The inverse problem, namely spectrum recovery
(which we study for the deformed Wigner case in Section \ref{sec:recovery}), has been well-studied
for covariance matrices. In \cite{karoui_spectrum} a convex optimization approach was used to
recover population spectra from samples. In \cite{kong2017spectrum}, the authors used a moment
method that works even in the sublinear regime where the dimension of the covariance matrix is much
larger than the number of samples. \cite{ledoit2015spectrum} used an approach that exploits the
natural discreteness of population spectra and suggested solving a nonlinear optimization problem
which essentially matches the empirical eigenvalues to the quantiles of the Marcenko-Pastur
distribution. The idea of natural discreteness of the population spectrum will also be useful for
the case of additive free-convolution with a semicircular distribution that we consider here.

Finally, from a Bayesian perspective shrinkage methods have been considered in other settings. In a
closely related problem in \cite{lexing_opaug} the authors suggested a Bayesian shrinkage method to
solve noisy elliptic systems of equations. For the case of covariance matrix estimation, linear
shrinkage is motivated by imposing am inverse Wishart prior, while other more sophisticated priors
give rise to nonlinear shrinkage methods (\cite{bergerprior},\cite{berger2020bayesian}).

\subsection{Organization of the Paper}

In Section \ref{sec:rmt} we review some well-known results from random matrix theory and present a
new result about trace functionals that involve both $A_n$ and $\hat{A}_n.$ These are going to be
the essential tools that we will need for the rest of the paper. In Section \ref{sec:algo} we derive
the oracle nonlinear shrinkage estimators for general continuous functions of $A$ and asymptotic
equivalents that are amenable to estimation. We also suggest an algorithm to perform asymptotically
optimal nonlinear shrinkage, when $H,\sigma$ are known. Section \ref{sec:recovery} considers the
problem of recovering $H,\sigma.$ Firstly, we show how $H$ can be recovered, given $\sigma,$ using a
nonlinear optimization problem and provide theoretical guarantees for consistency. We then explain
how $\sigma$ can be consistently estimated for a class of probability measures $H.$ In Section
\ref{sec:asympt_expand} we study asymptotic expansions of the shrinkers and the losses when
$\sigma\rightarrow 0$ and $\sigma\rightarrow\infty.$ Simulations and numerical experiments are
presented in Section \ref{sec:numerical}. Finally, Section \ref{sec:proofs} presents the complete
proofs of our results.

\section{Almost Sure Limits for a Class of Trace Functionals}\label{sec:rmt}

In this section we present some useful tools from random matrix theory. We start by introducing our
notation and stating well-known theorems. After that, we provide some new results about asymptotics
of trace functionals that include both $A$ and $\hat{A}$ which will be essential for justifying the
main algorithm in Section \ref{sec:algo}.

For a probability measure $\mu$ supported on the real line we will denote its Stieltjes transform by
$m_{\mu}(z)=\int (x-z)^{-1}\mu(dx),z\in\mathbb{C}^+.$ We will often omit the measure from the
subscript and just write $m(z),$ provided that it is clear which measure we are referring to. We
have the following well-known result, the so-called Wigner semicircle law (\cite{wigner_original}). 

\begin{theorem}[Theorem 2.4.2 in \cite{tao2012topics}]
Let $\left(M_{ij}\right)_{1\leq i,j}$ be mean 0, variance 1 real random variables such that $M_{ij}
= M_{ji}$ and $\left(M_{ij}\right)_{i<j}$ are independent and identically distributed. Then, the
spectral distribution of the sequence of random matrices
$n^{-1/2}M_n=\left(n^{-1/2}M_{ij}\right)_{1\leq i,j\leq n}$ converges weakly almost surely to the
Wigner semicircular distribution:
$$\mu_{sc} = \frac{\sqrt{(4-x^2)_+}}{2\pi}dx.$$
\end{theorem}

The above result gives the limiting spectral distribution of Wigner matrices. For the case of a
deformed Wigner matrix, such as $\hat{A}_n=A_n+\sigma n^{-1/2}Z_n,$ we have under the Assumptions \ref{assume:1} in
Subsection \ref{subsec:assum}:  

\begin{proposition}\label{prop:free_add_semi}
The matrix $\hat{A}_n$ has a limiting spectral distribution $\hat{\mu}$, which is a deterministic probability
measure with Stieltjes transform $m_{\hat{\mu}}(z)$ that satisfies: 
$$m_{\hat{\mu}}(z)=\int \frac{dH(t)}{t-z-\sigma^2m_{\hat{\mu}}(z)}.$$
\end{proposition}

This is the formula that describes the free additive
convolution $H\boxplus
\rho_{sc;\sigma^2}$ of a measure with a semicircular distribution (\cite{semi_free}). If $H=\delta_0,$ we
can solve for the Stieltjes transform $m_{\hat{\mu}}(z)$ in closed form and then use Stieltjes inversion to
recover the Wigner law. 

The first main contribution of this paper is to extend this result in the following theorem, which
is analogous to the results in \cite{ledoit2011eigenvectors} for the case of covariance matrices. As
in the case of covariance matrices, when \cite{ledoit2012nonlinear} used it to estimate a covariance
matrix using nonlinear shrinkage, this is going to be the main tool for theoretically justifying our
algorithms. In \cite{bouchaudrotational} a similar calculation is done using using replica symmetry
for matrices corrupted by orthogonally invariant noise.

\begin{theorem}\label{thm:basic_tool}
With the same assumptions as in Section \ref{sec:intro} we have for any $z\in\mathbb{C}^+$
\[
\frac{tr\left(h(A)\left(\hat{A}-z\right)^{-1}\right)}{n}\xrightarrow{a.s.}\int
\frac{h(t)dH(t)}{t-z-\sigma^2m_{\hat{\mu}}(z)}.
\]
Here, $m_{\hat{\mu}}(z)$ is the Stieltjes transform of the free additive convolution of $H$ with a
semicircular distribution of variance $\sigma^2,$ as in Proposition \ref{prop:free_add_semi}.
\end{theorem}

Although the theorem above was stated for a function $h$ that is continuous, it can be extended to cases with finitely many discontinuities which are not on atoms of the measure $H.$ In that case, taking $h(t)=\mathbb{I}_{[a,b]}$ gives the asymptotic overlap of the eigenvectors of $A,\hat{A},$ which the authors in \cite{bouchaudrotational} derived.

\section{Main Results}\label{sec:algo}

In this section, we motivate and present the main algorithm of the paper. We start by deriving an
oracle estimator that optimally approximates $h(A)$ among all rotationally invariant estimators. We
also find the optimal shrinker in closed form using the results from Section \ref{sec:rmt}. After
that, we explain how universality, namely the fact that in the large $n$ limit the distribution of
the noise does not affect the asymptotics we are interested in, can be used to simulate
approximately the optimally shrunk eigenvalues.

\subsection{Optimal Rotation Invariant Estimator}\label{subsec:ORIE}
We consider the spectral decomposition of $\hat{A},$ which has eigenvalues $\hat{\lambda}_1\geq
\cdots\geq \hat{\lambda}_n:$
\[
\hat{A} =\hat{W}\hat{\Lambda}\hat{W}^\intercal.
\]
For a continuous function $h$, an estimator $\Psi(\hat{A})$ of $h(A)$ is rotationally invariant if
$\Psi(O\hat{A}O^\intercal)=O\Psi(\hat{A})O^\intercal$ for any $n\times n$ orthogonal matrix $O.$
Searching for a rotationally invariant estimator of $h(A)$ seems reasonable, if we do not have any
prior information about the eigenstructure of $A.$ If such information was available, we might be
able to exploit it by approaching the problem in a Bayesian way. With that in mind, it also seems
reasonable to consider $\Psi(\hat{A})$ with the same eigenvectors as $\hat{A},$ such that
$\Psi(\hat{A})=\hat{W}{D}\hat{W}^\intercal.$ We are interested in choosing $\Psi$ to minimize the
Frobenius loss $\Norm{\Psi(\hat{A}) - h(A)}_F^2.$ We observe
that
\[
\Norm{\Psi(\hat{A})-h(A)}_F^2=\Norm{D - \hat{W}^\intercal h(A)\hat{W}}_F^2,
\]
which is minimized when
\begin{equation}\label{eq:oracle_d}
  D^{(h)}=(d_1^{(h)},\cdots,d_n^{(h)}) = \diag{\hat{W}^\intercal h(A)\hat{W}  }.
\end{equation} These clearly depends on the
unknown matrix $A$ and is not straightforward to estimate from the data. In the case $h(x) = x$ the
authors in \cite{bouchaudrotational}, \cite{bouchaud_book} show that the oracle quantities can be
asymptotically approximated by deterministic quantities that depend only on the limiting spectral
distribution of $\hat{A}$ and the noise $\sigma.$ The authors call this remarkable phenomenon the
large dimension miracle. It makes the oracle quantities amenable to estimation, for example via
kernel estimation. However, such a miracle does not seem very likely in the case of a general $h$
(and it is not entirely clear how to extend to the case of unknown $\sigma$). For example, already
for $h(x) = x^{-1},$ we will see that the optimal shrinkage is given by
\[
f_{1/t}^*(x) \equiv \frac{x+\sigma^2\int t^{-1}dH(t)}{(x+\sigma^2u(x))^2+\sigma^4v(x)^2},
\]
where $u(x)$ and $v(x)$ are the real and imaginary parts of $\lim _{\eta\rightarrow 0^+}m_{\hat{\mu}}({x+i\eta})$, i.e.,
$u(x)+iv(x)=\lim _{\eta\rightarrow 0^+}m_{\hat{\mu}}({x+i\eta})$. This already requires estimating
$m_H(0) = \int t^{-1}dH(t)$ and it is not hard to see that for other functions the situation can get
even more complicated.

\begin{definition}
For a continuous function $h$ on an open interval that contains the support of ${H},$ we define the functions $u_h,v_h:\mathbb{R}\rightarrow \mathbb{R}$ by $$u_h(x) + i v_h(x) = \lim_{\eta\downarrow 0} \int \frac{h(t)dH(t)}{t - x - i\eta -\sigma^2m_{\hat{\mu}}(x+i\eta)}.$$
\end{definition}

\begin{remark}
  The limit above exists because $\lim_{\eta\downarrow 0}m_{\hat{\mu}}(x+i\eta)$ exists (\cite{semi_free}).
\end{remark}

Below derive the optimal shrinker for a general continuous function $h.$ 
\begin{theorem}\label{thm:opt_shrinkage_der}
  Among all bounded continuous functions $f$ on an open interval containing
  $supp(H\boxplus\rho_{sc;\sigma^2})$ and the eigenvalues of $\hat{A}_n,$ the minimizer $f_h^*(x)$
  of the asymptotic quantity
  \[
  \lim_{n\rightarrow\infty}n^{-1}\Norm{f(\hat{A}_n)-h(A_n)}_F^2
  \]
  is given by $f_h^*(x) = v_h(x)/v(x)$ for $x\in supp(H\boxplus\rho_{sc;\sigma^2}).$
\end{theorem}

According to Theorem \ref{thm:basic_tool}, the measure
$n^{-1}\sum_{i=1}^nd_i^{(h)}\delta_{\hat{\lambda}_i}$ converges weakly almost surely to a measure
with density $\pi^{-1}v_h.$ This suggests that the asymptotic analog of the oracle quantities
$d_i^{(h)}$ is the quantity $v_h/v$ derived above. As an immediate corollary of Theorem
\ref{thm:opt_shrinkage_der} we have the following:

\begin{corollary}\label{corol:some_closed}
  \begin{enumerate}
  \item For the choice $h(t)=t$,
    \begin{equation}
      u_{t}(x) + i v_t(x) = \lim _{z=x+i\eta,\eta\downarrow 0}1+zm_{\hat{\mu}}(z)+\sigma^2m_{\hat{\mu}}(z)^2.
    \end{equation}
    This gives the optimal shrinkage function $f^*_t(x)=x+2\sigma^2u(x)$.
    
  \item For the choice $h(t)=1/t$,
    \begin{equation}
      u_{1/t}(x) + i v_{1/t}(x) = \lim _{z=x+i\eta,\eta\downarrow 0}\frac{m_{\hat{\mu}}(z)-m_H(0)}{z+\sigma^2m_{\hat{\mu}}(z)}.
    \end{equation}
    This gives the optimal shrinkage function
    \[
    f^*_{1/t}(x)=\frac{x+\sigma^2m_H(0)}{(x+\sigma^2u)^2+\sigma^4v(x)^2}.
    \]
  \item For the choice $h(t)=t^2$,
    \begin{equation}
      \begin{split}
        &u_{t^2}(x) + i v_{t^2}(x) = \lim_{z=x+i\eta, \eta\downarrow 0}\int \frac{t^2dH(t)}{t-z-\sigma^2m_{\hat{\mu}}(z)}\\
        &=\lim _{z=x+i\eta,\eta\downarrow 0}\int tdH(t)+z+\sigma^2m_{\hat{\mu}}(z)+m_{\hat{\mu}}(z)\left[z+\sigma^2m_{\hat{\mu}}(z)\right]^2\\
        &=\int t dH(t) + x +\sigma^2(u(x)+iv(x))+(u(x)+iv(x))\left[x+\sigma^2(u(x)+iv(x))\right]^2.
    \end{split}
    \end{equation}
    This gives the optimal shrinkage function
    \[
    f^*_{t^2}(x)=\sigma^2 +(x+\sigma^2u(x))^2-\sigma^4v^2(x)+2\sigma^2 u(x)(x+\sigma^2u(x)).
    \]
  \end{enumerate}
\end{corollary}

\begin{remark}
  Using Theorem \ref{thm:opt_shrinkage_der} we can show that for estimating $A_n^k$ in Frobenius norm
  we need the first $(k-2)$ moments of the measure $H$ for $k\geq 3.$
\end{remark}

\subsubsection{Pseudoinverses and Regularized Pseudoinverses}

We study the optimal shrinkage to estimate $A\left(A^2+\lambda ^2 I_n\right)^{-1}.$ If
$\lambda\downarrow 0,$ this converges to the pseudoinverse of the matrix $A.$ Using our usual
notation we have $h=h(t;\lambda)=t/(t^2+\lambda^2).$ This gives

\begin{equation}\begin{split}
&\int \frac{h(t)dH(t)}{t-z-\sigma^2m_{\hat{\mu}}(z)}=\int \frac{tdH(t)}{(t^2+\lambda^2)(t-z-\sigma^2m_{\hat{\mu}}(z))}\\
&=\int \left[\frac{z+\sigma^2m_{\hat{\mu}}(z)}{(z+\sigma^2m_{\hat{\mu}}(z))^2+\lambda^2}\frac{1}{t-z-\sigma^2m_{\hat{\mu}}(z)}\right]dH(t)\\
&+\frac{1}{2(\lambda i - z - \sigma^2m_{\hat{\mu}}(z))}\int \frac{dH(t)}{t-\lambda i} -\frac{1}{2(\lambda i + z +\sigma^2m_{\hat{\mu}}(z))}\int \frac{dH(t)}{t+\lambda i}\\
&=\frac{m_{\hat{\mu}}(z)(z+\sigma^2m_{\hat{\mu}}(z))}{(z+\sigma^2m(z))^2+\lambda^2}+\frac{m_H(\lambda i)}{2(\lambda i - z - \sigma^2m_{\hat{\mu}}(z))}-\frac{m_H(-\lambda i)}{2(\lambda i + z +\sigma^2m_{\hat{\mu}}(z))}
\end{split}
\end{equation}
which allows us to compute the optimal shrinkage as a function of $u,v,m_{H}(\lambda i).$



For the case of the pseudoinverse of a Hermitian matrix $A,$ we examine the following scenario. We
assume that there exist fixed $\delta>0,p\in (0, 1)$ such that $A$ has $a_n$ eigenvalues equal to 0,
$n-a_n$ eigenvalues greater than $\delta$ and $a_n/n\xrightarrow{a.s.}{p}$ as $n\rightarrow \infty.$
In that case we can write $H=p\delta_0+(1-p)\nu, $ where $\nu$ is a probability measure with support
contained in $[\delta,\infty).$ Under these assumptions the pseudoinverse of $A$ can be written as a
  function $h(A),$ where $h$ is continuous on $[0,\infty],$ $h(x)=1/x$ for $x\geq \delta$ and
  $h(x)=0$ in an open set containing $0$. We find in this case the Stieljes transform of
  $H\boxplus\rho_{sc;\sigma^2}$ satisfies \begin{equation}\label{eq:pseudo_stielt}
    m_{\hat{\mu}}(z)=-\frac{p}{z+\sigma^2m_{\hat{\mu}}(z)}+(1-p)\int \frac{d\nu(t)}{t-z-\sigma^2m_{\hat{\mu}}(z)}.
\end{equation}
Using this we see that

\begin{equation}
  \begin{split}
    &\int \frac{h(t)dH(t)}{t-z-\sigma^2m_{\hat{\mu}}(z)}\\
    &=(1-p)\int \frac{d\nu(t)}{t(t-z-\sigma^2m_{\hat{\mu}}(z))}=\frac{1-p}{z+\sigma^2m_{\hat{\mu}}(z)}\int\left[ \frac{1}{t-z-\sigma^2m_{\hat{\mu}}(z)}-\frac{1}{t}\right]d\nu(t)\\
    &=\frac{1-p}{z+\sigma^2m_{\hat{\mu}}(z)}\left[\frac{m_{\hat{\mu}}(z)+\frac{p}{z+\sigma^2m_{\hat{\mu}}(z)}}{1-p}-m_{\nu}(0)\right]\\
    &=\frac{m_{\hat{\mu}}(z)}{z+\sigma^2m_{\hat{\mu}}(z)}+\frac{p}{(z+\sigma^2m_{\hat{\mu}}(z))^2}-\frac{(1-p)m_{\nu}(0)}{z+\sigma^2m_{\hat{\mu}}(z)}.
  \end{split}
\end{equation}

\begin{remark}
  \begin{enumerate}
  \item If $p=0,$ the last formula reduces to $$\frac{m_{\hat{\mu}}(z)-m_{\nu}(0)}{z+\sigma^2m_{\hat{\mu}}(z)}=\frac{m_{\hat{\mu}}(z)-m_{H}(0)}{z+\sigma^2m_{\hat{\mu}}(z)}.$$ This, as expected, agrees with Corollay \ref{corol:some_closed}.
  \item When $H=p\delta_0+(1-p)\nu,$ we have $$m_H(\lambda i)=-\frac{p}{\lambda i}+(1-p)m_{\nu}(\lambda i)$$ and $$m_H(-\lambda i)=\frac{p}{\lambda i}+(1-p)m_{\nu}(-\lambda i).$$ Using these it is straightforward to see that the optimal shrinkage for the regularized pseudoinverse converges to the optimal shrinkage for the pseudoinverse as $\lambda\downarrow 0.$
  \end{enumerate}
\end{remark}

\subsection{Monte-Carlo Nonlinear Shrinkage}

We are now going to present an algorithm to approximate the oracle quantities. Based on Theorem \ref{thm:opt_shrinkage_der}, it is natural to try to compute $H$ and then solve for $u_h,v_h$. Our algorithm does not require solving numerically the equation for
the Stieltjes transform of the additive free convolution of $H$ with a semicircular distribution,
which can be tricky (\cite{rao_numerical}). We think that the
general idea behind it is likely to be applied in more complicated cases, in particular in problems
that do not have simple formulas for the optimal shrinkage as derived in Theorem
\ref{thm:opt_shrinkage_der}. The key observation is that the asymptotic equivalents
of the oracle quantities only depend on $H$ and are universal for all noise distributions. Hence,
although $A, Z$ are unknown, it is possible to replicate the asymptotic equivalents to the oracles
using a Monte-Carlo simulation.

Suppose that we know $\sigma, \lambda_1,\cdots,\lambda_n,$ or estimates
$\Tilde{\sigma},\Tilde{\lambda}_1,\cdots,\Tilde{\lambda}_n$ of those are available. The topic of
finding suitable choices for $\Tilde{\sigma}$ and $\Tilde{\lambda}_i$ for $1\leq i \leq n$ is going
to be the topic of the next section, as suggested by Theorem \ref{thm:recovery1}. Then, we suggest
the following simple procedure in Algorithm \ref{algo:MC_shrink} for approximately optimal nonlinear
shrinkage of the eigenvalues of $\hat{A}$ to estimate $h({A})$ in Frobenius norm. The complexity of
the algorithm is $\mathcal{O}(Kn^3).$ Notice that we use the notation $GOE(n)$ for the Gaussian
Orthogonal Ensemble in $\mathbb{R}^{n\times n}$ (\cite{tao2012topics}).

\begin{algorithm}[!t]
 \caption{MC Nonlinear Shrinkage}
 \label{algo:MC_shrink}
 \begin{algorithmic}[1]
   \State Inputs: $\Tilde{\sigma}, \Tilde{\lambda}_1,\cdots,\Tilde{\lambda}_n$ and a positive integer $K.$
   
   \For{$k=1,\cdots,K$}

   \State  Generate $\hat{Z}_k\sim \Tilde{\sigma} n^{-1/2}GOE(n).$

   \State Find the eigenvectors $\hat{g}_{1,k},\cdots,\hat{g}_{n,k}$ of
   $diag(\Tilde{\lambda}_1,\cdots,\Tilde{\lambda}_n)+\hat{Z}_k$ such that $\hat{g}_{i,k}$ corresponds to
   the $i-$th largest eigenvalue.

   \State Set $\hat{d}_{i,k}=\hat{g}_{i,k}^\intercal diag(h(\Tilde{\lambda}_1),\cdots,h(\Tilde{\lambda}_n)) \hat{g}_{i,k}.$

   \EndFor

   \State Output: $d_i^*=K^{-1}\sum_{k=1}^K\hat{d}_{i,k},1\leq i\leq n.$
 \end{algorithmic}
\end{algorithm}

Algorithm \ref{algo:MC_shrink} approximates the oracle nonlinear shrinkage in the following sense.
\begin{theorem}\label{thm:mc_average}
For a bounded continuous function $h$ defined on an open set that contains the support of $H,$ let
$d_i^{(h)}$ be the oracle quantities defined in \eqref{eq:oracle_d} Subsection \ref{subsec:ORIE} and $d_{i}^*$ the
output of the MC Nonlinear Shrinkage algorithm with input
$\Tilde{\sigma},\Tilde{\lambda}_1,\cdots,\Tilde{\lambda}_n,K\geq 1.$ Assume that
$\Tilde{\sigma}\rightarrow\sigma$ and
$n^{-1}\sum_{i=1}^n\delta_{\Tilde{\lambda}_i}\xrightarrow{a.s.}H$. Then, for any $a,b\in [0,1]:$
\[
\frac{\sum_{[na]}^{[nb]} d_i^{(h)}}{n}-\frac{\sum_{[na]}^{[nb]} d_i^*}{n}\xrightarrow{a.s.}{0}.
\]
\end{theorem}

\subsection{Different Loss Functions}

So far we have been interested in the case of Frobenius loss. For some applications other losses
might be more suitable. For this reason we shortly present how our results can be used to derive the
optimal nonlinear shrinkage for some other choices of losses. Some of the losses we consider here
(and many others) were studied for spiked covariance models in \cite{donoho_gavish_johnstone}.
Below we will be interested in the following losses:
\begin{enumerate}
    \item Stein loss: $L^{st}(A,B) = tr(A^{-1}B-I)-\log{\det{A^{-1}B}}.$
    \item Divergence Loss: $L^{div}(A,B)=tr(A^{-1}B-I)+tr(B^{-1}A-I).$
    \item The loss $L(A,B)=\Norm{A^{-1}B-I}_F^2.$
    
\end{enumerate}

\begin{proposition}\label{prop:diff_losses}
  Assume that (using the notation from the Assumptions in Section \ref{sec:intro}) $h_1>0.$ For any positive and bounded continuous function $f$ defined on an open set that eventually contains the eigenvalues of $\hat{A}$ we have almost surely:

  \begin{enumerate}
  \item   For the Stein loss $L^{st}(A,f(\hat{A}))$ we have:
  $$\lim_{n\rightarrow\infty} n^{-1} L^{st}(A,f(\hat{A}))=\int f(x)\frac{v_{1/t}(x)}{\pi}dx+\int \log t dH(t)-\int \log{f(x)}\frac{v(x)}{\pi}dx.$$ This is minimized for $f(x)=v(x)/v_{1/t}(x)=1/f_{1/t}^*(x).$
  \item For the Stein loss $L^{st}(f(\hat{A}),A)$ we have: $$\lim_{n\rightarrow\infty}n^{-1}L^{st}(f(\hat{A}),A)=\int \frac{1}{f(x)}\frac{v_t(x)}{\pi}dx+\int \log f(x)\frac{v(x)}{\pi}dx-\int \log t dH(t)-1.$$
    This is minimized for $f(x)=v_t(x)/v(x)=f_t^*(x).$
  \item For the divergence loss $L^{div}(A,f(\hat{A}))$ we have: $$\lim_{n\rightarrow\infty}n^{-1}L^{div}(A,f(\hat{A}))=\int f(x)\frac{v_{1/t}(x)}{\pi}dx+\int \frac{1}{f(x)}\frac{v_t(x)}{\pi}dx-2.$$
    This is minimized for $f(x)=\sqrt{v_t(x)/v_{1/t}(x)}=\sqrt{f_t^*(x)/f_{1/t}^*(x)}.$
    
  \item For the loss $L(A,f(\hat{A}))$ we have: $$\lim_{n\rightarrow\infty} n^{-1}L(A,f(\hat{A}))=1-2\int f(x)\frac{v_{1/t}(x)}{\pi}dx+\int f^2(x)\frac{v_{1/t^2}(x)}{\pi}dx .$$
    This is minimized for $f(x)=v_{1/t}(x)/v_{1/t^2}(x)=f_{1/t}^*(x)/f_{1/t^2}^*(x).$
  \item For the loss $L(f(\hat{A}),A)$ we have: $$\lim_{n\rightarrow\infty} n^{-1}L(f(\hat{A}),A)=1-2\int \frac{1}{f(x)}\frac{v_{t}(x)}{\pi}dx+\int \frac{1}{f^2(x)}\frac{v_{t^2}(x)}{\pi}dx .$$
    This is minimized for $f(x)=v_{t^2}(x)/v_{t}(x)=f_{t^2}^*(x)/f_{t}^*(x).$

\end{enumerate}
\end{proposition}

\section{Recovery of the Limiting Spectral Distribution}\label{sec:recovery}

So far we have assumed the we know $\sigma,H.$ In practice this is rarely true. Here we explain how
those can be consistently estimated. First of all, assume that $\sigma$ is known. If $\sigma$ is
unknown, we are going to see shortly that the problem can be ill-posed and further assumptions are
needed to guarantee recovery of $\sigma, H.$

\subsection{Spectrum Recovery: known noise level}\label{subsec:known_noise}

When $\sigma$ is known, we suggest the procedure in Algorithm \ref{algo:H_rec} that uses an
optimization problem for recovering the eigenvalues of $A.$

\begin{algorithm}
  \caption{Population Eigenvalues Recovery}
  \label{algo:H_rec}
  \begin{algorithmic}[1]
    \State Inputs:  $\hat{\lambda}_1,\cdots, \hat{\lambda}_n, \sigma .$ 

    \State Sample $\hat{Z} \sim \sigma GOE(n).$

    \State For $T=(t_1,\cdots,t_n)^\intercal\in\mathbb{R}^n$ with $t_1\geq \cdots t_n$, denote by
    $\hat{t}_1\geq \cdots \geq\hat{t_n}$ the eigenvalues of $\diag{T} + n^{-1/2}\hat{Z}.$
   
    \State Solve the optimization problem $T^* = \argmin_T n^{-1} \sum_{j=1}^n(\hat{t}_j-\hat{\lambda}_j)^2.$

    \State Output $T^*$.
  \end{algorithmic}
\end{algorithm}

To minimize the objective above we suggest using the BFGS algorithm. A reasonable choice of a starting point that suggest is a point with independent Gaussian coordinates centered at the sample mean of the spectral distribution of $\hat{A}$. The optimization can be done quickly due to the fact that the gradients of the loss are easy to find in closed form. In particular, we have the following immediate proposition, which shows that the spectral decomposition of $T+\hat{Z}$ contains all the essential information to perform a BFGS update:

\begin{proposition}
Using the notation from Algorithm \ref{algo:H_rec}, if $T+\hat{Z}=\sum_{j=1}^n \hat{t}_j\hat{x}_j\hat{x}_j^\intercal$ is the spectral decomposition of $T+n^{-1/2}\hat{Z},$ we have for all $i=1,\cdots,n:$

$$\partial _{t_i}\hat{t}_j=\hat{x}_{ij}^2.$$ By $\hat{x}_{ij}$ we denote the $i$-th coordinate of $\hat{x}_j\in\mathbb{R}^n.$
\end{proposition}

\begin{proof}
Let $E_{ii}\in\mathbb{R}^{n\times n}$ be the diagonal matrix with $i$-entry 1 and all other entries 0. Let $M_i(s) = T+n^{-1/2}\hat{Z}+sE_{ii}.$ We have $\frac{d}{ds}M(s)=E_{ii},$ so using the \textit{Hadamard first variation formula} (Page 57, \cite{tao2012topics}), we get $$\partial_{t_i}\hat{t}_j = \hat{x}_j^\intercal E_{ii}\hat{x}_j=\hat{x}_{ij}^2.$$
\end{proof}

We have the following results that justify using this procedure:

\begin{theorem}\label{thm:recovery1}
Under the assumptions from Section \ref{sec:intro}, we have:
\begin{enumerate}
    \item $$\min_T\frac{1}{n}\sum_{i=1}^n\left(\hat{t}_i-\hat{\lambda}_i\right)^2\xrightarrow{a.s.}0.$$
    \item If $T^*$ is a minimizer of the optimization problem above with $t_1^*\geq\cdots\geq t_n^*,$ then
    
    $$\frac{1}{n}\sum_{i=1}^n(t_i^*-\lambda_i)^2\xrightarrow{a.s.}0.$$
\end{enumerate}
\end{theorem}

\begin{remark}
  \begin{enumerate}
  \item In the optimization problem we use only one copy of $\hat{Z}\sim\sigma GOE(n).$ In the
    high-dimensional limit $n\rightarrow\infty$ this is enough. Alternatively, as a regularization
    step, we could use multiple copies and solve the optimization problem repeatedly, getting
    solutions $T_1^*,\cdots,T_K^*.$ We can then return $T^*=K^{-1}\sum_{i=1}^KT_i^*.$
  \item Theorem \ref{thm:recovery1} shows that
    $n^{-1}\sum_{i=1}^n\delta_{t_i^*}\xrightarrow{a.s.}H.$ This implies that the estimated
    eigenvalues can be used as input to Algorithm \ref{algo:MC_shrink} and the assumptions of
    Theorem \ref{thm:mc_average} will be satisfied.
  \end{enumerate}
\end{remark}

\subsection{Spectrum Recovery: unknown noise level}
If $\sigma$ is unknown, it is impossible to recover the measure $H$ simply by observing the free
additive convolution with a semicircular measure of variance $\sigma^2.$ To see why, assume that $H$
is semicircular with variance $s^2.$ Then, $\mu_{H,\sigma^2}$ is semicircular with variance
$s^2+\sigma^2$ and it is impossible to separate the semicircular components of this measure. We
conclude that further assumptions are needed. In fact, it is clear from the discussion above that
only probability measures that cannot be written as the free additive convolution of a semicircular
distribution and another probability measure are candidates for exact asymptotic recovery. For this
reason, we are going to impose the following assumption throughout this section.

\begin{assumption}
The measure $H$ cannot be written as the free additive convolution of a semicircular distribution with positive variance and a probability measure.
\end{assumption}

In that case, if we solve the optimization problem from \ref{subsec:known_noise} for a choice $\hat{\sigma}<\sigma,$ Theorem \ref{thm:recovery1} suggests that the output will recover $H\boxplus\rho_{sc;\sigma^2-\hat{\sigma}^2},$ while the objective should converge to 0. If we solve for a choice $\hat{\sigma}>\sigma,$ then it is impossible to make the objective tend to $0.$ In particular, we have the following:

\begin{proposition}\label{prop:unknown_sigma}
Let $R_n(\hat{\sigma})$ be the optimal value of the objective of the optimization problem in Algorithm \ref{algo:H_rec} with $\sigma$ substituted by $\hat{\sigma}$. Then:
\begin{enumerate}
    \item For $\hat{\sigma} <\sigma,$ $\limsup{R_n(\hat{\sigma})} = 0.$
    \item For $\hat{\sigma} > \sigma,$ $\liminf{R_n(\hat{\sigma})} > 0.$
\end{enumerate}
\end{proposition}


Proposition \ref{prop:unknown_sigma} indicates that we can use a scree plot -type method to determine the noise level $\sigma.$ In particular, we can solve the problem for several choices of the noise level and choose $\sigma$ before the objective becomes significantly  larger than 0. This is going to be illustrated in Section \ref{sec:numerical}.

\section{Asymptotic Expansions}\label{sec:asympt_expand}

We study the asymptotic expansions of the oracle quantities and the optimal shrinkage functions in the regimes of "large noise" ($\sigma\rightarrow \infty$) and "small noise" ($\sigma\rightarrow  0).$

\subsection{The Large Noise Asymptotics}

If $\sigma\rightarrow\infty,$ we have the following:

\begin{proposition}\label{prop:large_noise}
\begin{enumerate}
    \item If $Z\sim GOE(n),$ the oracle quantities $d_i^{(h)}$ defined in \eqref{eq:oracle_d} almost surely satisfy:
    
    $$\lim_{n\rightarrow\infty}\lim_{\sigma\rightarrow\infty}\max_{1\leq i\leq n}\abs{d_i^{(h)}-\int h(t)dH(t)}=0.$$
    
    \item The optimal shrinkage $f_h^*(x)$ satisfies $$\lim_{\sigma \rightarrow\infty}f_h^*(\sigma x)=\int h(t)dH(t)$$ for $\abs{x}<2.$
\end{enumerate}
\end{proposition}

\begin{remark}
  Proposition \ref{prop:large_noise} shows that in the regime of very large $\sigma,$ the optimal
nonlinear shrinkage quantities for estimation of $h(A)$ in Frobenius norm are essentially constant
and achieve mean-squared-error equal to $\Var{h(H)}.$ This is reasonable, as an extremely large
$\sigma$ should make estimation of $h(A)$ extremely hard. Notice that for $\sigma\rightarrow\infty$
the eigenvalues of $\hat{A}$ scale almost linearly with $\sigma$ and the limiting spectral
distribution of $\sigma^{-1}\hat{A}$ is the semicircle law, which is indeed supported on $[-2,2].$
\end{remark}

\subsection{The Small Noise Asymptotics}

We now study the regime $\sigma\rightarrow 0.$ Since for $\sigma=0$ the eigenvectors of $A$ may not
be uniquely determined, we assume for simplicity in this subsection that $A$ has distinct
eigenvalues. In that case we have for $\sigma\rightarrow 0:$

\begin{proposition}\label{prop:small_noise}

If $h\in C^1(\mathbb{R})$ and $Z\sim GOE(n):$ 
\begin{enumerate}
\item The oracle quantities for $\sigma\rightarrow 0$ satisfy: $$\lim_{\sigma\rightarrow 0}\max_{1\leq i\leq n}\frac{\abs{d_i^{(h)}-h(\lambda_i)}}{\sigma}=0.$$
  
\item $$\lim_{n\rightarrow\infty}\lim_{\sigma\rightarrow
  0}\frac{\Norm{h(\hat{A})-h(A)}_F^2}{n\sigma^2}=\iint
  \frac{(h(t)-h(s))^2}{(t-s)^2}dH(t)dH(s).$$
  
\end{enumerate}

\end{proposition}

We see from Proposition \ref{prop:small_noise} that the mean-squared-error grows sublinearly in
$\sigma^2$ for the optimal nonlinear shrinkage, if $\sigma$ is small, while using no shrinkage gives
mean squared error $\approx \sigma^2 \iint (h(t)-h(s))^2/(t-s)^2dH(t)dH(s)$ for $\sigma $
small. This is because, as we see from part 1 of Proposition \ref{prop:small_noise} in the Gaussian
case, the oracle quantities converge to $h(\lambda_i)$ fast for $\sigma\rightarrow 0.$

\section{Numerical Experiments}\label{sec:numerical}

\subsection{Experiments for Algorithm \ref{algo:H_rec}}
Here we consider three examples.

\textbf{Example 1.} Firstly we check the effectiveness of the deconvolution algorithm (Algorithm
\ref{algo:H_rec}). For $H=(\delta_1+\delta_4+\delta_9)/3,\sigma^2=1$ and 20 equally spaced values of
$n$ (starting from $n=50$ and ending with $n=1000$) we solve the optimization problem described in
Algorithm \ref{algo:H_rec}. We start from 10 randomly initialized points and keep the stationary
point of the objective that leads to the smallest value. We plot in Figure \ref{fig:deconv1}, as a
function of $n,$ the resulting normalized mean squared error, which we define
as $$\frac{\frac{1}{n}\sum_{i=1}^n(t_i^*-\lambda_i)^2}{\Var{H}}.$$ We also present the recovered
eigenvalues $t_i^*$ versus $i=1,\cdots,n$ for the values $n=250,500,750,1000.$

\begin{figure}
    \centering
    \begin{subfigure}[b]{0.49\textwidth}
    \includegraphics[width=\textwidth]{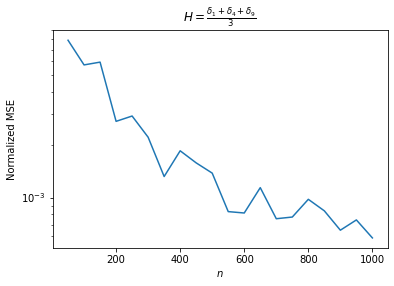}
    \label{fig:deconv1.1}
    \caption{Algorithm \ref{algo:H_rec} Objective}
    \end{subfigure}
    \hfill
    \begin{subfigure}[b]{0.47\textwidth}
    \centering
    \includegraphics[width=\textwidth]{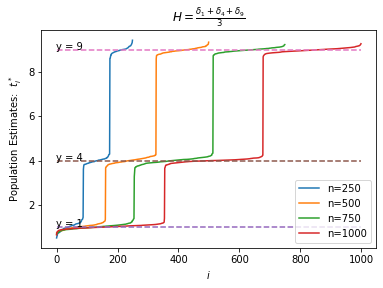}
    \label{fig:deconv1.2}
    \caption{Recovered Eigenvalues}

    \end{subfigure}
    \caption{Algorithm \ref{algo:H_rec} Experiment for $H=(\delta_1+\delta_4+\delta_9)/3$}
    \label{fig:deconv1}
\end{figure}

\textbf{Example 2.}  For a more complicated choice of spectral distribution $H$ we design the
following experiment. We consider 200 randomly sampled points from circles centered at 0 with radii
0.5 and 1 respectively (presented with red and blue dots in the plot below). We add Gaussian noise
with standard deviation 0.05 to the data. After generating those points, labeled as
$x_1,\cdots,x_{200}\in\mathbb{R}^2,$ we build the connectivity matrix $A\in\mathbb{R}^{200\times
  200}$ using the Gaussian kernel: $$A_{ij}=\exp{\left(-\frac{\Norm{x_i-x_j}^2}{2h^2}\right)}.$$
Here we choose $h=0.1.$ We assume that we have access only to a
matrix $$\hat{A}=A+\sqrt{\frac{2}{200}}Z,$$ where $Z$ is a standard Gaussian Wigner matrix. This
corresponds to the choice $\sigma^2=2.$ We use Algorithm \ref{algo:H_rec} to estimate the
eigenvalues of $A.$ Below we plot the sample eigevalues (that is the eigenvalues of $\hat{A}$), the
true eigenvalues of $A$ and, finally, the estimated eigenvalues from the deconvolution algorithm. We
see in Figure \ref{fig:deconv2} that the reconstruction is very close.

\begin{figure}
    \centering
    \begin{subfigure}[b]{0.49\textwidth}
    \includegraphics[width=\textwidth]{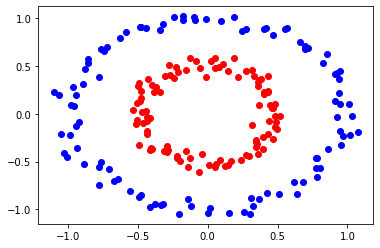}
    \label{fig:deconv2.1}
    \caption{$n=200$ Generated Data Points}
    \end{subfigure}
    \hfill
    \begin{subfigure}[b]{0.47\textwidth}
    \centering
    \includegraphics[width=\textwidth]{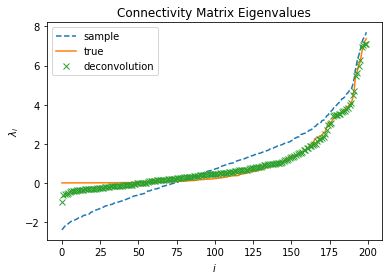}
    \label{fig:deconv2.2}
    \caption{Eigenvalues}

    \end{subfigure}
    \caption{Algorithm \ref{algo:H_rec} Experiment for the Connectivity Matrix created using a Gaussian Kernel with $h=0.1.$}
    \label{fig:deconv2}
\end{figure}

\textbf{Example 3.} Finally, we consider an example with unknown noise level $\sigma^2.$ In
particular, we consider $H=(\delta_5+\delta_{10})/2,\sigma^2=1,n=200.$ We take $Z$ to have entries
drawn from a Laplace distribution. This time $\sigma^2$ is unknown, so we have to use several
choices $\hat{\sigma}$ in the optimization problem and choose the largest $\hat{\sigma}$ for which
the objective is close to 0. Figure \ref{fig:deconv3} indicates using $\hat{\sigma}^2$ from 0.951 to
1.029. Refining the grid can give us an even closer estimate. We solve the optimization problem for
$\hat{\sigma}^2=0.99,$ which is the midpoint between the two values of $\sigma^2$ from above.

\begin{figure}
    \centering
    \begin{subfigure}[b]{0.49\textwidth}
    \includegraphics[width=\textwidth]{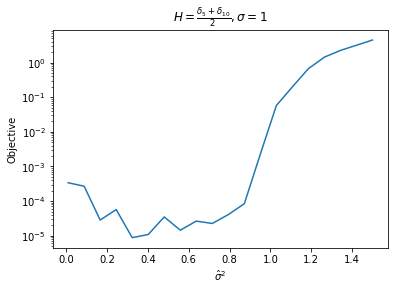}
    \label{fig:deconv3.1}
    \caption{Algorithm \ref{algo:H_rec} Objective versus $\hat{\sigma}^2$}
    \end{subfigure}
    \hfill
    \begin{subfigure}[b]{0.47\textwidth}
    \centering
    \includegraphics[width=\textwidth]{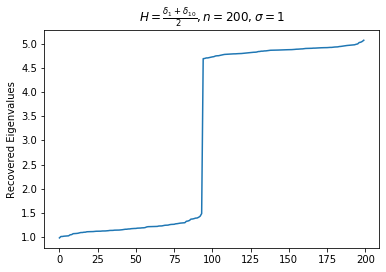}
    \label{fig:deconv3.2}
    \caption{Eigenvalues for the estimated $\hat{\sigma}^2=0.99$}

    \end{subfigure}
    \caption{Algorithm \ref{algo:H_rec} Experiment for $H=(\delta_1+\delta_{10})/2,\sigma=1.$ Here $\sigma$ is unknown and is estimated from $\hat{A}.$}
    \label{fig:deconv3}
\end{figure}



\subsection{Noisy Linear Systems of Equations.}

The first application we consider is the following. We want to solve a linear system of equations of
the form $Ax=b,$ whose solution we denote $x^*=A^{-1}b.$ The matrix $A$ is unknown. Instead we have
access to a noisy estimate $\hat{A}=A+\sigma n^{-1/2}Z,$ where $A,Z$ satisfy the assumptions from
Section \ref{sec:intro}.  Solving $\hat{A}x=b$ gives $x=\hat{A}^{-1}b.$ The problem is that
$\hat{A}^{-1}$ might be a very bad estimate of $A^{-1}$ and ill-conditioned. For this reason we
suggest using $x^{(f)}=f(\hat{A})b,$ where $f$ is a bounded continuous function on $[h_1,h_2].$ Our
goal is to choose $f$ to minimize $$\lim_{n\rightarrow\infty}\frac{\Norm{x^{(f)}-x^*}^2}{n}.$$ We
study two different distributional assumptions on $b.$

\begin{enumerate}
\item $b\sim \mathcal{N}(0,I_n)$. In that case using Lemma \ref{lp_lemma} we see that:
  
  \begin{equation}\begin{split}
      \lim_{n\rightarrow\infty}\frac{\Norm{x^{(f)}-x^*}^2}{n}=\lim_{n\rightarrow\infty}\frac{\Norm{(f(\hat{A})-A^{-1})b}^2}{n}=\lim_{n\rightarrow\infty}\frac{\Norm{f(\hat{A})-A^{-1}}_F^2}{n},
    \end{split}
  \end{equation}
  which is minimized for $f(x)=f^*_{1/t}(x).$
\item $b=Ax^*,x^*\sim \mathcal{N}(0,I_n)$. Similarly using Lemma \ref{lp_lemma} we see that:
  
  \begin{equation}\begin{split}
      \lim_{n\rightarrow\infty}\frac{\Norm{x^{(f)}-x^*}^2}{n}=\lim_{n\rightarrow\infty}\frac{\Norm{(f(\hat{A})A-I_n)x^*}^2}{n}=\lim_{n\rightarrow\infty}\frac{\Norm{f(\hat{A})A-I_n}_F^2}{n}.
    \end{split}
  \end{equation}
  
  Using exactly the same argument as in Proposition \ref{prop:diff_losses} we see that the limit
  is almost surely $$1-2\int f(x)\frac{v_{t}(x)}{\pi}dx+\int f^2(x)\frac{v_{t^2}(x)}{\pi}dx,$$
  which is minimized for $f(x)=f_{t}^*(x)/f_{t^2}^*(x).$
\end{enumerate}

For $H=(\delta_1+\delta_{10})/2,n=500$ we plot in Figure \ref{fig:syst1} the normalized
mean-squared-error (which we define as $\Norm{x^{(f)}-x^*}^2/\Norm{x^*}^2$) for several values of
$\sigma^2.$

\begin{figure}
    \centering
    \begin{subfigure}[b]{0.49\textwidth}
    \includegraphics[width=\textwidth]{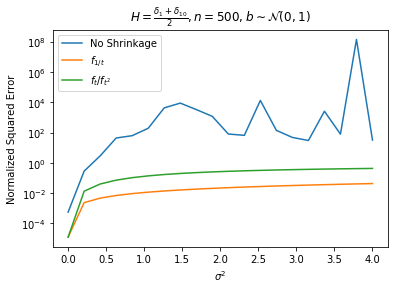}
    \label{fig:syst1.1}
    \caption{$b\sim\mathcal{N}(0,I_n)$}

    \end{subfigure}
    \hfill
    \begin{subfigure}[b]{0.49\textwidth}
    \centering
    \includegraphics[width=\textwidth]{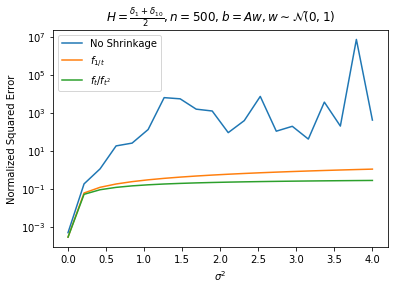}
    \label{fig:syst1.2}
    \caption{$x^*\sim\mathcal{N}(0,I_n)$}

    \end{subfigure}
    \caption{$\Norm{x^{(f)}-x^*}^2/\Norm{x^*}^2$ for different choices of $f$ and $H=(\delta_1+\delta_{10})/2.$}
    \label{fig:syst1}
\end{figure}

We repeat the experiment for $H=(\delta_1+\delta_4+\delta_9)/3,n=200.$ The results can be seen in
Figure \ref{fig:syst2}.

\begin{figure}
    \centering
    \begin{subfigure}[b]{0.49\textwidth}
    \includegraphics[width=\textwidth]{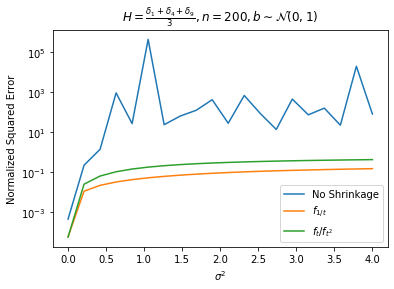}
    \label{fig:syst2.1}
    \caption{$b\sim\mathcal{N}(0,I_n)$}

    \end{subfigure}
    \hfill
    \begin{subfigure}[b]{0.49\textwidth}
    \centering
    \includegraphics[width=\textwidth]{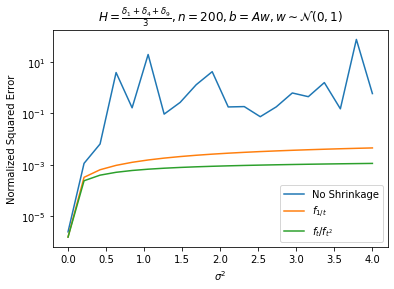}
    \label{fig:syst2.2}
    \caption{$x^*\sim\mathcal{N}(0,I_n)$}

    \end{subfigure}
    \caption{$\Norm{x^{(f)}-x^*}^2/\Norm{x^*}^2$ for different choices of $f$ and $H=(\delta_1+\delta_4+\delta_9)/3.$}
    \label{fig:syst2}
\end{figure}

In both cases we see that $\hat{A}$ becomes eventually ill-conditioned, if $\sigma^2$ increases. As
expected, if $b\sim\mathcal{N}(0,I_n)$ the first shrinkage outperforms the second at all noise
levels, while for $x^*\sim\mathcal{N}(0,I_n)$ the opposite is true.

\subsection{Experiments for Algorithm \ref{algo:MC_shrink}}

We consider the problem of estimating $A,A^{-1}$ and $\sqrt{A}$ in Frobenius norm, when we only have
access to $\hat{A}.$ For $n=500$ and several values of $\sigma$ we generate $\hat{A}=A+\sigma
n^{-1/2}Z,$ where $Z$ is a standard Gaussian Wigner matrix. Here $A$ is chosen as a diagonal matrix
with diagonal entries chosen uniformly at random from $\{1,4,9\}.$ We plot for
$h(t)=t,h(t)=1/t,h(t)=\sqrt{t}$ the oracle error and the error that can be achieved by using
Algorithm \ref{algo:H_rec} to recover the eigenvalues of $A$ and Algorithm \ref{algo:MC_shrink} with
$K=1$ to perform nonlinear shrinkage. We see that in all cases the error achieved by our algorithm
is very close to the oracle. For $\sigma ^2$ large, notice that the problem of eigenvalue recovery
for $A$ becomes increasingly harder, hence the error in estimation of $\lambda_1,\cdots,\lambda_n$
increases. This can lead to problems for the function $h(t)=1/t$ which is unbounded near 0, hence we
clip all the recovered eigenvalues that we get from Algorithm \ref{algo:H_rec} to be at least
0.3. Notice that for $h(t)=1/t,h(t)=\sqrt{t}$ we do not plot the no shrinkage

\begin{figure}
    \centering
    \begin{subfigure}[b]{0.3\textwidth}
    \includegraphics[width=\textwidth]{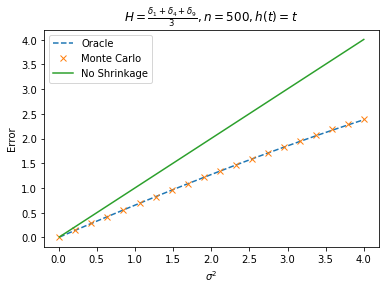}
    \label{fig:estx}
    \caption{Algorithm \ref{algo:MC_shrink} for $h(t)=t.$}
    \end{subfigure}
    \hfill
    \begin{subfigure}[b]{0.3\textwidth}
    \centering
    \includegraphics[width=\textwidth]{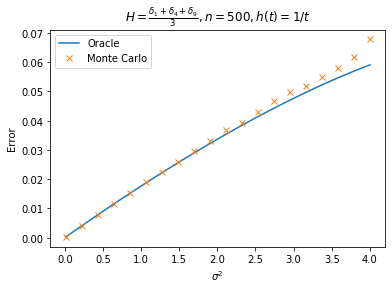}
    \label{fig:est_inv}
    \caption{Algorithm \ref{algo:MC_shrink} for $h(t)=1/t.$}

    \end{subfigure}
    \hfill
    \begin{subfigure}[b]{0.3\textwidth}
    \centering
    \includegraphics[width=\textwidth]{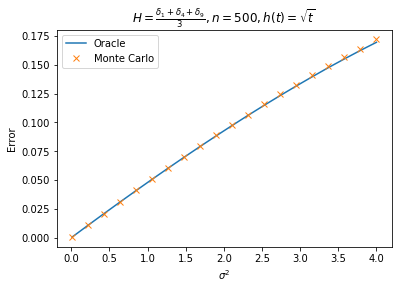}
    \label{fig:est_sqrt}
    \caption{Algorithm \ref{algo:MC_shrink} for $h(t)=\sqrt{t}.$}

    \end{subfigure}
    \caption{Algorithm \ref{algo:MC_shrink} Experiment for $H=(\delta_1+\delta_{4}+\delta_9)/3.$ We plot the oracle and shrinkage errors versus $\sigma$ for $h(t)=t,1/t,\sqrt{t}.$}
    \label{fig:Est}
\end{figure}




\section{Proofs}\label{sec:proofs}

\subsection{Proofs for Section \ref{sec:rmt}}

We start by presenting a well-known lemma for the tails of a standard Gaussian random variable.
\begin{lemma}\label{lem:gaussian_tail}
For any $M>0$ and $Z\sim\mathcal{N}(0,1)$ we have $$\mathbb{P}(\abs{Z}>M)\leq 2M^{-1}\exp{(-x^2/2)}.$$
\end{lemma}
\begin{proof}
We have \begin{equation}\begin{split}
    \mathbb{P}(\abs{Z}>M)=2\mathbb{P}(Z>M)=2\int _M^\infty \exp{(-x^2/2)}dx\\\leq 2\int_M^\infty\frac{x}{M}\exp{(-x^2/2)}dx=2M^{-1}\exp{(-x^2/2)}.
\end{split}
\end{equation}
\end{proof}
We will need the following lemma which is adapted from Lemma 7.8, Lemma 7.9 and Lemma 7.10 from
\cite{yau_dynamical}.
\begin{lemma}\label{lp_lemma}
Let $q\geq 2$ and $X_1,\cdots,X_N,Y_1,\cdots,Y_N$ be independent random variables with mean 0,
variance 1 and $2q$-th moment bounded by $c_0$. Then, for any deterministic $(b_i)_{1\leq i\leq N},
(a_{ij})_{1\leq i,j\leq N}$ we have for some positive constant $C_q=C_q(c_0)$:

\begin{equation}
\Norm{\sum_i b_i (X_i^2-1)}_q\leq C_q (\sum_i \abs{b_i}^2)^{\frac{1}{2}}
\end{equation}

\begin{equation}
\Norm{\sum_{i,j}a_{ij}X_iY_j}_q \leq C_q (\sum_{i,j}a_{ij}^2)^{\frac{1}{2}}
\end{equation}

\begin{equation}
\Norm{\sum_{i\neq j}a_{ij}X_iX_j}_q\leq C_q (\sum_{i\neq j}a_{ij}^2)^{\frac{1}{2}}
\end{equation}
\end{lemma}

\begin{proof}[Proof of Theorem \ref{thm:basic_tool}]
The proof involves two main steps.

\begin{enumerate}
    \item \textbf{Step 1:} Show that the theorem holds for if $Z\sim  GOE(n).$
    \item \textbf{Step 2:} Reduce the problem to the case of bounded random variables as entries of $Z.$
    \item \textbf{Step 3:} Show that the results are universally true and independent of the distribution of $Z$ as long as $Z,A$ are asymptotically free.
\end{enumerate}

\textbf{Step 1:} For $Z \sim GOE(n)$ (which is invariant under conjugation by an orthogonal matrix) it is enough to consider the case of diagonal matrix $A.$ If $A=\diag{A_1,\cdots,A_n},$ then using the Schur complement formula we have that the $i$-th diagonal entry of $h(A)\left(\hat{A}-z\right)^{-1}$ is given by $$h(A_i)/\left(A_{i}-z-\sigma n^{-1/2}Z_{ii}-\sigma^2 n^{-1}a_i^\intercal\left(A_{-i}+n^{-1/2}Z_{-i}-z\right)^{-1}a_i\right).$$ Here $A_{-i}$ is the $(n-1)\times (n-1)$ matrix that we get if we omit the $i$-th element $A_i$ of $A,$ $a_i\in\mathbb{R}^{(n-1)}$ the $i$-th row of $Z$ if we omit the diagonal element. 

We have from Lemma \ref{lem:gaussian_tail} for any fixed $\epsilon>0$ and $n\geq \epsilon^{-2}:$ $$\mathbb{P}\left(\max_{1\leq i\leq n}\abs{Z_{ii}}>\epsilon\sqrt{n}\right)\leq n\mathbb{P}\left(\abs{Z_{ii}}>\epsilon\sqrt{n}\right)=\mathcal{O}\left(n\exp{(-n\epsilon^2/2)}\right).$$ Since $\epsilon$ was arbitrary, we conclude by the Borel-Cantelli lemma that $$\max_{1\leq i\leq n}\abs{Z_{ii}}/\sqrt{n}\xrightarrow{a.s.}0.$$

Similarly, by Lemma \ref{lp_lemma} for $q=3$, $$\max_{1\leq i\leq n}n^{-1}\abs{ a_i^\intercal\left( A_{-i}+\sigma n^{-1/2}Z_{-i}-z\right)^{-1}a_i - tr\left(\left(A_{-i}+\sigma n^{-1/2}Z_{-i}-z\right)^{-1}\right)}.$$ Using the Cauchy interlacing formula (\cite{tao2012topics}), we see that it must also be true that

\begin{equation}\begin{split}
\max_{1\leq i\leq n}\abs{n^{-1}tr\left(\left(A_{-i}+\sigma n^{-1/2}Z_{-i}-z\right)^{-1}\right)-m_{\hat{\mu}}(z)}\xrightarrow{a.s.}0,
\end{split}
\end{equation}

as for fixed $z\in\mathbb{C}^+$ the differences $$\abs{tr\left(\left(A_{-i}+\sigma n^{-1/2}Z_{-i}-z\right)^{-1}\right)-tr\left(\left(A+\sigma n^{-1/2}Z-z\right)^{-1}\right)(z)}$$ are going to be uniformly bounded (due to the interlacing phenomenon).

As a consequence, we see that $$n^{-1}tr\left(h(A)\left(\hat{A}-z\right)^{-1}\right)=\omicron{(1)}+\sum_{i=1}^n\frac{h(A_i)}{t_i-z-\sigma^2m_{\hat{\mu}}(z)},$$
 which proves the result for $z\in\mathbb{C}^+$ and $Z\sim  GOE(n).$

\textbf{Step 2:} Fix $M>0.$ For this step we assume, in order to slightly simplify the formulas, that without loss of generality that $\sigma=1.$ Define $Z_{i,j}^{(M)}=Z_{ij}\mathbb{I}_{\abs{Z_{ij}}<M}.$ We also define $\hat{A}^{(M)} = A + n^{-1/2}Z^{(M)}.$ We have 
\begin{equation}\label{eq:reduce_Z}\begin{split}\abs{n^{-1}tr\left(h(A)(\hat{A}^{(M)}-z)^{-1}\right)
-n^{-1}tr\left(h(A)(\hat{A}-z)^{-1} \right)}\\
=\frac{\abs{tr\left(h(A)(\hat{A}-z)^{-1}(Z-Z^{(M)})(\hat{A}^{(M)}-z)^{-1}\right)}}{n\sqrt{n}}\\=\frac{\abs{tr\left((\hat{A}^{(M)}-z)^{-1}h(A)(\hat{A}-z)^{-1}(Z-Z^{(M)})\right)}}{n\sqrt{n}}\\
\leq \frac{\Norm{(\hat{A}^{(M)}-z)^{-1}h(A)(\hat{A}-z)^{-1}}_{op}\Norm{(Z-Z^{(M)})}_F}{n}.
\end{split}
\end{equation}

Notice that here we have used the fact that for two $n\times n$ matrices $M_1,M_2$ we have $$\abs{tr\left(M_1M_2\right)}\leq \Norm{M_1}_F\Norm{M_2}_F\leq \sqrt{n}\Norm{M_1}_{op}\Norm{M_2}_F,$$ where the first inequality follows from Cauchy-Schwartz in $\mathbb{R}^{n\times n}$ and the second one from the fact that the Frobenius norm of a real matrix is the $l_2-$norm of its singular values.

We conclude from \eqref{eq:reduce_Z} that for a fixed bounded continuous function and a fixed complex number $z$ in the upper half-plane we have \begin{equation}\begin{split}\abs{n^{-1}tr\left(h(A)(\hat{A}^{(M)}-z)^{-1}\right)
-n^{-1}tr\left(h(A)(\hat{A}-z)^{-1} \right)}=\mathcal{O}\left(n^{-1}\Norm{Z-Z^{(M)}}_F\right).
\end{split}
\end{equation}
We now observe that $$\mathbb{E}\left[\Norm{\left(Z-Z^{(M)}\right)}_F^2\right]=n^2\mathbb{E}\left[Z_{11}^2;\abs{Z_{11}}\geq M\right]$$ and $$\Var{\Norm{Z-Z^{(M)}}_F^2}=\mathcal{O}\left(n^2\mathbb{E}\left[Z_{11}^4;\abs{Z_{11}}>M\right]\right).$$  Fix any $\epsilon>0$ and take $M$ large enough such that $\mathbb{E}[Z_{11}^2;\abs{Z_{11}}>M]\leq \epsilon^2/2$ and $\mathbb{E}\left[Z_{11}^4;\abs{Z_{11}}>M\right]\leq 1.$ Then we have \begin{equation}\label{eq:bound_Z_diff}\begin{split}&\mathbb{P}\left(n^{-1}\Norm{Z-Z^{(M)}}_F>\epsilon\right)=\mathbb{P}\left(n^{-2}\Norm{Z-Z^{(M)}}_F^2>\epsilon^2\right)\\
&\leq\mathbb{P}\left(n^{-2}\Norm{Z-Z^{(M)}}_F^2-\mathbb{E}\left[n^{-2}\Norm{Z-Z^{(M)}}_F^2\right]>\epsilon^2/2\right)=\mathcal{O}\left(n^{-2}\epsilon^{-4}\right).\end{split}\end{equation}
Using the Borel-Cantelli lemma we conclude that. almost surely, $n^{-1}\Norm{Z-Z^{(M)}}_F\leq \epsilon$ eventually. 
Using \eqref{eq:reduce_Z} we see that for $M$ large enough we have eventually almost surely
\begin{equation}\begin{split}\abs{n^{-1}tr\left(h(A)(\hat{A}^{(M)}-z)^{-1}\right)
-n^{-1}tr\left(h(A)(\hat{A}-z)^{-1} \right)}\\\leq \epsilon \Norm{(\hat{A}^{(M)}-z)^{-1}h(A)(\hat{A}-z)^{-1}}_{op}\leq \epsilon Im(z)^{-2}\Norm{h}_{\infty}.
\end{split}\end{equation}

To finish this step, we define  $\mu_M=\mathbb{E}\left[Z_{ij}^{(M)}\right],\sigma_M=\sqrt{\Var{Z_{ij}^{(M)}}}$ and $\Tilde{Z}_{ij}=\left( Z_{ij}^{(M)}-\mu_M\right)/\sigma_M,$ which are random variables with mean 0 and variance 1. A similar argument shows that, if $\Tilde{A}=A+n^{-1/2}\Tilde{Z}$ and $M$ is large enough, then eventually almost surely we have \begin{equation}\label{eq:bound_M}
\abs{n^{-1}tr\left(h(A)\left(\hat{A}^{(M)}-z\right)^{-1}\right)-n^{-1}tr\left(h(A)\left(\Tilde{A}-z\right)^{-1}\right)}\leq\epsilon.
\end{equation}

To see why, using the same bound as in \eqref{eq:reduce_Z}, we see that \begin{equation}\label{eq:reduce_Z2}\begin{split}&\abs{n^{-1}tr\left(h(A)\left(\hat{A}^{(M)}-z\right)^{-1}\right)-n^{-1}tr\left(h(A)\left(\Tilde{A}-z\right)^{-1}\right)}\leq \\
&\frac{\Norm{\left(\hat{A}^{(M)}-z\right)^{-1}h(A)\left(\Tilde{A}-z\right)^{-1}}_{op}\Norm{\Tilde{Z}-Z^{(M)}}_F}{n}\leq n^{-1} \Norm{h}_{\infty}Im(z)^{-2}\Norm{\Tilde{Z}-Z^{(M)}}_F.
\end{split}
\end{equation}

It remains to bound $\Tilde{Z}-Z^{(M)}$ in Frobenius norm. Let $e=(1,\cdots,1)^\intercal\in\mathbb{R}^n.$ Then, $$\Tilde{Z}-Z^{(M)}=\Tilde{Z}^{(M)}(1-\sigma_M^{-1})-\frac{\mu_M}{\sigma_M}ee^\intercal.$$ Using this we get \begin{equation}\label{eq:reduce_Z_3}\begin{split}
\Norm{\Tilde{Z}-Z^{(M)}}_F^2=(1-\sigma_M^{-1})^2\Norm{Z^{(M)}}_F^2-2\frac{\mu_M}{\sigma_M} (1-\sigma_M^{-1})e^\intercal Z^{(M)}e+n^2\frac{\mu_M^2}{\sigma_M^2}\\
\leq n^2\frac{\mu_M^2}{\sigma_M^2}-2\frac{\mu_M}{\sigma_M}(1-\sigma_M^{-1})e^\intercal Z^{(M)}e+(1-\sigma_M^{-1})\left(\Norm{Z}_F+\Norm{Z-Z^{(M)}}_F\right)^2
\end{split}
\end{equation}
Now we know that:
\begin{enumerate}
\item $\lim_{M\rightarrow\infty}\mu_M=0$ and $\lim_{M\rightarrow\infty}\sigma_M=1$ from the dominated convergence theorem.
\item $n^{-1}\Norm{Z-Z^{(M)}}$ can be made arbitrarily small eventually (by choosing $M$ large enough), using \eqref{eq:bound_Z_diff} and the Borel-Cantelli lemma.
\item $n^{-1}\Norm{Z}_F\leq \sqrt{n^{-1}\Norm{Z}_{op}},$ which converges to $2$ almost surely (\cite{tao2012topics}).
\item Finally,
  \[
  n^{-2}e^\intercal Z^{(M)}e=\frac{\sum_{1\leq i,j\leq
      n}Z_{ij}\mathbb{I}_{\abs{Z_{ij}}<M}}{n^2}\xrightarrow{a.s.}\mathbb{E}\left[Z_{11};\abs{Z_{11}}<M\right]
  \]
  by the strong law of large numbers.
    
\end{enumerate}

Taking all of the above into consideration, we see that for $M$ large enough we see that for $M$ large enough we have eventually almost surely that the bound from \eqref{eq:bound_M} is true. Since $\Tilde{Z}$ is a Wigner ensemble with bounded entries, we have reduced the problem to the case of bounded random variables.

\textbf{Step 3:} We now show that under the assumptions in Subsection \ref{subsec:assum} the theorem
is also true. The idea is to show, using free-probabilistic tools, that if $\Tilde{Z}$ is a Wigner
ensemble with all moments finite, the the limit of the trace functionals of interest depends on the
noise distribution via only its first two moments. First of all, notice that it is enough to prove
the result for a polynomial $h$ and the extend to a general continuous function by a simple density
argument. As a result, it is enough to consider $h(t)=t^k,k\in\mathbb{N}$ and show that
$n^{-1}tr\left(A^k\left(\hat{A}-z\right)^{-1}\right)$ has a limit almost surely and the limit does
not depend on the distribution of $Z.$ Similarly, it is enough to show that for any $m\in\mathbb{N}$
the trace functional $n^{-1}tr\left(A^k\hat{A}^m\right)$ has a limit almost surely and the limit
does not depend on the distribution of $Z.$ Writing $\hat{A}^m=\left(A+\sigma n^{-1/2}Z\right)^m$
and expanding in monomial terms we see that $n^{-1}tr\left(A^k\hat{A}^m\right)$ is the sum of a
finite number of terms all of which have the form $n^{-1}tr\left(A^{n_1}(\sigma
n^{-1/2}Z)^{m_1}\cdots A^{n_s}\left(n^{-1/2}Z\right)^{m_s}\right)$ for some $s\geq 1$ and
nonnegative integers $n_1,m_1,\cdots,n_s,m_s.$ Since $Z$ is a Wigner ensemble and $A$ is independent
of $Z,$ we conclude that $A,n^{-1/2}Z$ are almost surely asymptotically free (Theorem 20 in
\cite{mingo2017free}). As a consequence, we have that all terms of the form
$n^{-1}tr\left(A^{n_1}(\sigma n^{-1/2}Z)^{m_1}\cdots A^{n_s}\left(n^{-1/2}Z\right)^{m_s}\right)$
converge almost surely and the limit depends only on the limiting spectral distributions of
$A,n^{-1/2}Z,$ which are given by $H$ and a semicircular distribution respectively. In particular,
the limit is independent of the distribution of $Z.$ We conclude that the limit is the same as with
the Gaussian assumption on $Z.$ This completes the proof.
\end{proof}

\subsection{Proofs for Section \ref{sec:algo}}

\begin{proof}[Proof of Theorem \ref{thm:opt_shrinkage_der}]

First of all, let $f$ be an analytic function on the complex plane. Then, we have using Cauchy's integral formula: $$\frac{tr\left(f(\hat{A})h(A)\right)}{n}=-\frac{1}{2\pi i }\oint_{\abs{z}=R} \frac{tr\left(h(A)\left(\hat{A}-z\right)^{-1}\right)}{n}f(z)dz,$$ where the integral is considered on a fixed circle centered at $0$ with radius $R$ such that eventually  $\Norm{\hat{A}}_{op}\leq R/2.$ Consider $$R_n(z)=\frac{tr\left(h(A)\left(\hat{A}-z\right)^{-1}\right)}{n}$$ and for a fixed $z$ let us denote by $R_{\infty}(z)$ the almost sure limit of $R_n(z)$ described in Theorem \ref{thm:basic_tool}. Then, for $n$ large enough we have almost surely that:

$$\abs{R_n(z)}=\abs{\frac{tr\left(h(A)\left(\hat{A}-z\right)^{-1}\right)}{n}}\leq \frac{\Norm{h(A)}_{op}}{(R-R/2)}=2R^{-1}\Norm{h(A)}_{op},$$

and $$\abs{R_n'(z)}=\abs{\frac{tr\left(h(A)\left(\hat{A}-z\right)^{-2}\right)}{n}}\leq \frac{\Norm{h(A)}_{op}}{(R-R/2)^2}=4R^{-2}\Norm{h(A)}_{op},$$ so on the circle $\{z\in\mathbb{C}:\abs{z}=R\}$ the sequence of functions $\{R_n(z)\}_{n\geq 1}$ almost surely consists of functions that are 
uniformly bounded and equicontinuous. Fix some $\epsilon>0$ and consider a finite subset $C\subset \{z\in\mathbb{C}:\abs{z}=R\}$ such that for any $z\in\mathbb{C}$ with $\abs{z}=R$ there exists $\Tilde{z}\in C$ such that $\abs{z-\Tilde{z}}<\epsilon.$ Since for any such $z,\Tilde{z}$ we have $$\abs{R_n(z)-R_{\infty}(z)}\leq \abs{R_n(z)-R_n(\Tilde{z})}+\abs{R_n(\Tilde{z})-R_{\infty}(\Tilde{z})}+\abs{R_{\infty}(\Tilde{z})-R_{\infty}(z)}$$ $$=\mathcal{O}\left(\epsilon + \sup_{\Tilde{z}\in C}\abs{R_n(\Tilde{z})-R_{\infty}(\Tilde{z})}\right),$$ we know that almost surely $$\limsup \sup_{\abs{z}=R}\abs{R_n(z)-R_{\infty}(z)}=\mathcal{O}(\epsilon).$$ Since $\epsilon$ was arbitrary we conclude that $R_n\xrightarrow{a.s.}R_{\infty}$ uniformly on $\{z\in\mathbb{C}:\abs{z}=R\}.$ Using this result we see that $$\frac{tr\left(h(A)f\left(\hat{A}\right)\right)}{n}\xrightarrow{a.s.}{-\frac{1}{2\pi i}}\oint_{\abs{z}=R} R_{\infty}(z)f(z)dz=-\frac{1}{2\pi i}\oint _{\Gamma_{\delta}}R_{\infty}(z)f(z)dz,$$ where we have changed the integral to be on a counterclockwise curve $\Gamma_{\delta}$ which we take to be a rectangle with vertices $\pm R\pm i \delta.$ Taking $\delta\downarrow 0$ we get 

\begin{equation}\begin{split}
&-\frac{1}{2\pi i}\oint R_{\infty}(z)f(z)dz\\
&=-\frac{1}{2\pi i}\lim_{\delta\downarrow 0}\left[\int _{-R}^R R_{\infty}(x-i\delta)f(x-i\delta) dx-\int _{-R}^R R_{\infty}(x+i\delta)f(x+i\delta)\right]\\
&=\int \frac{v_h(x)}{\pi}f(x)dx.
\end{split}
\end{equation}

In other words, we have shown that $$\frac{tr\left(f(\hat{A})h(A)\right)}{n}\xrightarrow{a.s.}\int \frac{v_h(x)}{\pi}f(x)dx$$ for $f$ analytic. Using a simple density argument we see that this result is actually true for any function $f$ that is continuous and bounded in an open set that contains the support of $\mu_{H,\sigma^2}$ and eventually all the eigenvalues of $\hat{A}.$

We now see that 

\begin{equation}\begin{split}
\frac{\Norm{f(\hat{A})-h(A)}_F^2}{n}=\frac{\Norm{f(\hat{A})}_F^2-2tr\left(f(\hat{A}h(A))+\Norm{h(A)}_F^2\right)}{n}\\ \xrightarrow{a.s.}\int f^2(x)\frac{v(x)}{\pi}dx-2\int \frac{v_h(x)}{\pi}f(x)dx+\int h^2(x)dH(x).
\end{split}
\end{equation}

Minimizing over $f$ we see that for $x$ in the support of $\mu_{H,\sigma^2}$ the minimizer satisfies $f^*(x)=v_h(x)/v(x).$ This completes the proof.
\end{proof}

\begin{proof}[Proof of Corollary \ref{corol:some_closed}]
\begin{enumerate}
    \item We have $$\int\frac{tdH(t)}{t-z-\sigma^2m_{\hat{\mu}}(z)}=1+\int\frac{z+\sigma^2m_{\hat{\mu}}(z)}{t-z-\sigma^2m_{\hat{\mu}}(z)}=1+(z+\sigma^2m_{\hat{\mu}}(z))m_{\hat{\mu}}(z),$$ where the last equality follows from Proposition \ref{prop:free_add_semi}.
    \item Using $$\frac{1}{t(t-z-\sigma^2m_{\hat{\mu}}(z))}=\frac{1}{z+\sigma^2m_{\hat{\mu}}(z)}\left[\frac{1}{t-z\sigma^2m_{\hat{\mu}}(z)}-\frac{1}{t}\right],$$ we have $$\int\frac{dH(t)}{t(t-z-\sigma^2m_{\hat{\mu}}(z))}=\frac{m_{\hat{\mu}}(z)-m_H(0)}{z+\sigma^2m_{\hat{\mu}}(z)},$$ so the formula for the asymptotically optimal shrinkage for $A^{-1}$ follows.
    \item We have

    \begin{equation}\begin{split}&\int\frac{t^2dH(t)}{t-z-\sigma^2m_{\hat{\mu}}(z)}\\
    &=\int \frac{t^2-(z+\sigma^2m_{\hat{\mu}}(z))^2}{t-z-\sigma^2m_{\hat{\mu}}(z)}dH(t)+(z+\sigma^2m_{\hat{\mu}}(z))^2\int \frac{dH(t)}{t-z-\sigma^2m_{\hat{\mu}}(z)}\\
    &=\int \left(t+z+\sigma^2m_{\hat{\mu}}(z)\right)dH(t)+(z+\sigma^2m_{\hat{\mu}}(z))^2m_{\hat{\mu}}(z)\\
    &=\int tdH(t)+z+\sigma^2m_{\hat{\mu}}(z)+(z+\sigma^2m_{\hat{\mu}}(z))^2m_{\hat{\mu}}(z).
    \end{split}
    \end{equation}

\end{enumerate}

\end{proof}

\begin{proof}[Proof of Theorem \ref{thm:mc_average}]
First of all, we observe that it is enough to prove the theorem for $K=1,$ where $\hat{Z}_1\sim\Tilde{\sigma} n^{-1/2}GOE(n).$ Hence we consider only that case and ignore the dependency on $k$ in the subscripts in Algorithm \ref{algo:MC_shrink}. We will write $\Tilde{\Lambda}=\diag{\Tilde{\lambda}_1,\cdots,\Tilde{\lambda}_n}$ for the diagonal matrix in Step 2 of Algorithm \ref{algo:MC_shrink}. We will denote by $\Tilde{m}_1\geq\cdots\geq \Tilde{m}_n$ the eigenvalues of $\Tilde{\Lambda}+\hat{Z}_1$. From Theorem \ref{thm:basic_tool} we know that: $$n^{-1}\sum_{i=1}^nd_i^*\delta_{\Tilde{m}_i}\xrightarrow{a.s.}\mu_h, $$ where $\mu_h$ is a finite measure with Stieltjes transform given by $$\int \frac{h(t)dH(t)}{t-z-\sigma^2m_{\hat{\mu}}(z)}.$$ In addition, we know that $n^{-1}\sum_{i=1}^n\delta_{\Tilde{m}_i}$ converges weakly almost surely to the additive free convolution of $H$ with a semicircular distribution with variance $\sigma^2,$ which is a probability measure $H\boxplus\rho_{sc;\sigma^2}$ without atoms (\cite{semi_free}). We conclude that for any $x_1,x_2\in\mathbb{R}$ we have $n^{-1}\sum_{i=1}^nd_i^*\mathbb{I}_{\Tilde{m}_i\in[x_1,x_2]}\xrightarrow{a.s.}\mu_h([x_1,x_2]).$ Similarly $n^{-1}\sum_{i=1}^nd_i^{(h)}\mathbb{I}_{\hat{\lambda}_i\in[x_1,x_2]}\xrightarrow{a.s.}\mu_h([x_1,x_2]).$ Since $H\boxplus\rho_{sc;\sigma^2}$ has no atoms, the proof is completed if we consider $x_1,x_2$ be the $a,b$-quantiles respectively of $\mu_{H,\sigma^2}$.

\end{proof}

\begin{proof}[Proof of Proposition \ref{prop:diff_losses}]

In the proof of Theorem \ref{thm:opt_shrinkage_der} we saw that, if $f$ satisfies the assumptions of
Proposition \ref{prop:diff_losses}, then almost surely
\begin{equation}\label{eq:trace_func_lem}
  \lim_{n\rightarrow\infty}n^{-1}tr\left(f(\hat{A})h(A)\right)=\int f(x) \frac{v_h(x)}{\pi}dx.
\end{equation}

\begin{enumerate}
\item
  \begin{equation}\begin{split}
      n^{-1}L^{st}(A,f(\hat{A})) = \frac{1}{n}tr\left(A^{-1}f(\hat{A})\right)-1+\frac{1}{n}\sum_{i=1}^n \log \lambda_i-\frac{1}{n}\sum_{i=1}^n\log f(\hat{\lambda_i})\\ \xrightarrow{a.s.}\int f(x)\frac{v_{1/t}(x)}{\pi}dx-1+\int \log t dH(t)-\int \log f(x) \frac{v(x)}{\pi}dx,
    \end{split}
  \end{equation}
  where we used \eqref{eq:trace_func_lem} and the fact that the spectrum of $A$ converges weakly almost surely to $H,$ while the spectrum of $\hat{A}$ converges weakly almost surely to the measure $H\boxplus\rho_{sc;\sigma^2}$ with density $v(x)/\pi.$ 
  Minimizing the integrand with respect to $f(x)$ for $x$ fixed is straightforward using derivatives and gives the desired result.
    \item \begin{equation}\begin{split}
    n^{-1}L^{st}(f(\hat{A}),A)= \frac{1}{n}tr\left(f(\hat{A})^{-1}A\right)-1+\frac{1}{n}\sum_{i=1}^n\log f(\hat{\lambda_i})-\frac{1}{n}\sum_{i=1}^n \log \lambda_i\\ \xrightarrow{a.s.}\int \frac{1}{f(x)}\frac{v_t(x)}{\pi}dx-\int \log t dH(t)+\int \log f(x) \frac{v(x)}{\pi}dx-1.
    \end{split}
    \end{equation}
    Minimizing with respect to $f$ is again straightforward.
    \item Using \eqref{eq:trace_func_lem} we get: \begin{equation}\begin{split}
    n^{-1}L^{div}(A,f(\hat{A}))=\frac{tr\left(Af(\hat{A})^{-1}\right)}{n}+\frac{tr\left(A^{-1}f(\hat{A})\right)}{n}-2\\ \xrightarrow{a.s.}\int f(x)\frac{v_{1/t}(x)}{\pi}dx+\int \frac{1}{f(x)}\frac{v_{t}(x)}{\pi}dx-2.
    \end{split}
    \end{equation}
    \item \begin{equation}\begin{split}
    n^{-1}\Norm{A^{-1}f(\hat{A})-I}_F^2=\frac{tr\left(A^{-2}f(\hat{A})^2\right)}{n}-2\frac{tr\left(A^{-1}f(\hat{A})\right)}{n}+1\\ \xrightarrow{a.s.}\int f^2(x)\frac{v_{1/t^2}(x)}{\pi}dx-2\int f(x)\frac{v_{1/t}(x)}{\pi}dx+1.
    \end{split}
    \end{equation}
    
    \item \begin{equation}\begin{split}
    n^{-1}\Norm{Af(\hat{A})^{-1}-I}_F^2=\frac{tr\left(A^2f(\hat{A})^{-2}\right)}{n}-2\frac{tr\left(Af(\hat{A})^{-1}\right)}{n}+1\\ \xrightarrow{a.s.}1-2\int \frac{1}{f(x)}\frac{v_t}{\pi}dx+\int \frac{1}{f^2(x)}\frac{v_{t^2}(x)}{\pi}dx.
    \end{split}
    \end{equation}

\end{enumerate}
\end{proof}

\subsection{Proofs for Section \ref{sec:recovery}}

\begin{proof}[Proof of Theorem \ref{thm:recovery1}]
  
  \begin{enumerate}
  \item We take $t_i=\lambda_i.$ Then, $$n^{-1}\sum_{i=1}^n\delta_{\hat{t}_i}\xrightarrow{a.s.}H\boxplus\rho_{sc;\sigma^2}.$$
    In addition, $$n^{-1}\sum_{i=1}^n\delta_{\hat{\lambda}_i}\xrightarrow{a.s.}H\boxplus\rho_{sc;\sigma^2}.$$
    
    Finally, applying Weyl's inequality ((1.54) in \cite{tao2012topics}) to $\hat{A}=A+\sigma n^{-1/2}Z$ and using the fact that for any $\epsilon>0$ the eigenvalues of $n^{-1/2}Z$ eventually lie in $[-2-\epsilon,2+\epsilon]$ almost surely,  we have $\lambda_i-2\sigma+\omicron{(1)}\leq \hat{\lambda}_i\leq \lambda_i+2\sigma+\omicron{(1)},$ so that $\hat{\lambda}_i$ are almost surely uniformly bounded. Similarly for $\hat{t}_i.$ We conclude that for this choice of $t_i$'s the 2-Wasserstein distance of $$n^{-1}\sum_{i=1}^n(\hat{t}_i-\hat{\lambda}_i)^2\xrightarrow{a.s.}0.$$ The proof is completed.
    \item If we denote by $\nu_n=n^{-1}\sum_{i=1}^n\delta_{t_i^*}$ the probability measure that corresponds to the solution to the optimization problem in Algorithm \ref{algo:H_rec}, then we know that $\nu_{n}$ is tight sequence of probability measures. To see why, by Weyl's eigenvalue inequality and the fact that almost surely  the largest eigenvalue of $n^{-1/2}Z$ tends to 2 and the largest eigenvalue of $n^{-1/2}Z$ tends to -2 we get $t_i^*-2\sigma+\omicron{(1)}\leq \hat{t}_i^*\leq t_i^*+2\sigma+\omicron{(1)},$ so 
    \begin{equation}\label{eq:tightness}
    \abs{t_i^*-\hat{\lambda}_i}\leq \abs{\hat{t}_i^*-\hat{\lambda}_i}+2\sigma +\omicron{(1)}.
    \end{equation}
    For any large $M>0$ fixed we now see from \eqref{eq:tightness} that \begin{equation}\begin{split}
    M^2\frac{\abs{\{i:1\leq i\leq n, \abs{t_i-\lambda_i}>M\}}}{n}\leq n^{-1}\sum_{i=1}^n(t_i-\hat{\lambda}_i)^2\mathbb{I}\{\abs{t_i-\hat{\lambda}_i}>M\}\\
    \leq n^{-1}\left(\sqrt{\sum_{i=1}^n(\hat{t}_i^*-\hat{\lambda}_i)^2}+2\sigma \sqrt{n}+\omicron{(\sqrt{n})}\right)^2=\left(R_n({\sigma})+2\sigma+\omicron{(1)}\right)^2.
    \end{split}
    \end{equation}
    On the other hand, we know that any weak subsequential limit of $\nu_n$ has to be equal to $H,$ as from the previous part of the theorem we know that $n^{-1}\sum_{i=1}^n\delta_{t_i^*}\boxplus\rho_{sc;\sigma^2}$ converges weakly to $H\boxplus\rho_{sc;\sigma^2}.$ We conclude that $\nu_n\xrightarrow{\mathcal{D}}H$ almost surely. 
   Fix $\epsilon>0$ and consider $M>0$ large enough (to be determined later). For the moment we assume that $[-M/2,M/2]$ contains $[h_1,h_2]$ from Assumption 3 in section \ref{sec:intro}.  In addition, we assume that $\abs{\hat{\lambda}_i}\leq M/2$ for all $i.$ 
   
   
   Then, using the triangle inequality we have:
   \begin{equation}\begin{split}
   \sqrt{n^{-1}\sum_{i=1}^n(t_i^*)^2\mathbb{I}\{\abs{t_i^*}>M\}}\leq \sqrt{n^{-1}\sum_{i=1}^n(\hat{t}_i^*-\hat{\lambda}_i)^2}+\sqrt{n^{-1}\sum_{i=1}^n\hat{\lambda}_i^2\mathbb{I}\{\abs{t_i^*}>M\}}\\+\sqrt{n^{-1}\sum_{i=1}^n(\hat{t}_i^*-t_i^*)^2\mathbb{I}\{\abs{t_i^*}>M\}}\\ \leq \sqrt{n^{-1}\sum_{i=1}^n(\hat{t}_i^*-\hat{\lambda}_i)^2} + (\max_{1\leq i\leq n}\abs{\hat{\lambda}_i}+2\sigma +\omicron{(1)})\sqrt{\frac{\abs{i:1\leq i\leq n, \abs{t_i^*}>M}}{n}}\\ \leq \sqrt{n^{-1}\sum_{i=1}^n(\hat{t}_i^*-\hat{\lambda}_i)^2} + (\max_{1\leq i\leq n}\abs{\hat{\lambda}_i}+2\sigma +\omicron{(1)})\sqrt{\frac{\abs{i:1\leq i\leq n, \abs{t_i^*-\hat{\lambda}_i}>M/2}}{n}}
   \end{split}
   \end{equation}

   The first term in the above inequality goes to 0 almost surely, as we saw in part (1), while the second term can be made arbitrarily small from the bound in \eqref{eq:tightness}. So if we choose $M$ large enough, then eventually almost surely we have
   \begin{equation}\label{eq:2nd_mom}
   n^{-1}\sum_{i=1}^n(t_i^*)^2\mathbb{I}\{\abs{t_i^*}>M\}\leq \epsilon.
   \end{equation}
   
   Combining the bound in \eqref{eq:2nd_mom}, since $\epsilon$ was arbitrary, with the fact that $\nu_n\xrightarrow{\mathcal{D}}H$ and $n^{-1}\sum_{i=1}^n\delta_{\lambda_i}\xrightarrow{\mathcal{D}}H$ we see that $$n^{-1}\sum_{i=1}^n(t_i^*-\lambda_i)^2\xrightarrow{a.s.}0.$$

    
\end{enumerate}
\end{proof}

\begin{proof}[Proof of Proposition \ref{prop:unknown_sigma}]

\begin{enumerate}
    \item If $\hat{\sigma}<\sigma,$ then taking $t_i$ to be the $(i-1) / p$ quantile of $H\boxplus \rho_{sc;\sigma^2-\hat{\sigma}^2}$ (which is the additive free-convolution of $H$ with a semicircular distribution with variance $\sigma^2-\hat{\sigma}^2$) gives that the empirical distribution $n^{-1}\sum_{i=1}^n\hat{t}_i$ converges weakly almost surely to $$H\boxplus \rho_{sc;\sigma^2-\hat{\sigma}^2}\boxplus \rho_{sc;\hat{\sigma}^2}=H\boxplus \rho_{sc;\sigma^2}.$$ As a consequence, the Wasserstein 2-distance of $n^{-1}\sum_{i=1}^n\hat{t}_i$ and $n^{-1}\sum_{i=1}^n\hat{\lambda}_i$ converges almost surely to 0, so $$n^{-1}\sum_{i=1}^n(\hat{t}_i-\hat{\lambda}_i)^2\xrightarrow{a.s.}0.$$
    This shows that almost surely $\limsup{R_n}(\hat{\sigma})=0.$
    \item Fix $\hat{\sigma}>\sigma$ and consider the event $\mathcal{A}=\{\liminf{R_{\hat{\sigma}}}=0\}.$ Assume that $A$ has positive probability. Then, for an $\omega\in \mathcal{A}$ there exists a sequence $n_k\uparrow\infty$ such that $R_{n_k}(\hat{\sigma};\omega)\rightarrow 0.$
    
    If we denote by $\nu_n=n^{-1}\sum_{i=1}^n\delta_{t_i^*(\omega)}$ the probability measure that corresponds to the solution to the optimization problem in Algorithm \ref{algo:H_rec}, then we know that $\nu_{n_k}$ is tight sequence of probability measures, as in the proof of Theorem \ref{thm:recovery1}. 
    We conclude that there exists a subsequence of $\{\nu_{n_k}\}_{k\geq 1}$ that converges weakly to a probability measure $\Tilde{H}.$ Then, we must have, due to the fact that $R_{n_k;\omega}(\hat{\sigma})\rightarrow 0,$ $$\Tilde{H}\boxplus \rho_{sc;\hat{\sigma}^2}=H\boxplus\rho_{sc;\sigma^2}.$$
\end{enumerate}

Since $\Tilde{H}\boxplus\rho_{sc;\hat{\sigma}^2}=\Tilde{H}\boxplus\rho_{sc;\hat{\sigma}^2\boxplus\sigma^2}\boxplus\rho_{sc;\sigma^2},$ we deduce that $H=\Tilde{H}\boxplus\rho_{sc;\hat{\sigma}^2-\sigma^2}.$ This is a contradiction, so $\mathbb{P}(\mathcal{A})=0$ and the proof is completed.
\end{proof}



\subsection{Proofs for Section \ref{sec:asympt_expand}}
\begin{proof}[Proof of Proposition \ref{prop:large_noise}]
\begin{enumerate}
    \item For $n$ fixed and $\sigma\rightarrow\infty,$ $\sigma^{-1}\hat{A}_n=\sigma^{-1}A_n+n^{-1/2}Z_n\rightarrow n^{-1/2}Z_n$ and the eigenvectors of $\hat{A}_n$ converge to the eigenvectors of $Z_n$ which are uniformly distributed with respect to the Haar measure. Let us denote by $z_1,\cdots,z_n$ the $l_2-$normalized eigenvectors of $Z.$  We have $$\mathbb{E}\left[(z_i^\intercal w_j)^2|A\right]=\frac{1}{n}\Rightarrow{\mathbb{E}}\left[z_i^\intercal Az_i|A\right]=\frac{tr\left(A\right)}{n}.$$ Applying Theorem 5.1.4 in \cite{vershynin2018high} for the function $f(X)=X^\intercal A X$ (which is Lipschitz on the unit sphere with Lipschitz constant 2$\Norm{A}_{op}$) we have that there exists a constant $C>0$ such that for any $\epsilon>0$ and any $i:$
    \begin{equation}
    \mathbb{P}\left(\abs{z_i^\intercal Az_i-n^{-1}tr\left(A\right)}>\epsilon\right)\leq 2\exp{\left(-\frac{cn\epsilon^2}{\Norm{A}_{op}^2}\right)}\leq 2\exp{\left(-\frac{cn\epsilon^2}{h_2^2}\right)}.
    \end{equation}
    
    Using the union bound we have  \begin{equation}
    \mathbb{P}\left(\max_{1\leq i \leq n}\abs{z_i^\intercal Az_i-n^{-1}tr\left(A\right)}>\epsilon\right)\leq 2n\exp{\left(-\frac{cn\epsilon^2}{h_2^2}\right)}.
    \end{equation}
    
    The Borel-Cantelli lemma implies that almost surely we eventually have $$\max_{1\leq i\leq n}\abs{z_i^\intercal Az_i-n^{-1}tr\left(A\right)}\leq\epsilon.$$ Since $\epsilon$ was arbitrary we have $$\max_{1\leq i\leq n}\abs{z_i^\intercal Az_i-n^{-1}tr\left(A\right)}\xrightarrow{a.s}0.$$ To finish the proof we see that $$n^{-1}tr(A)=n^{-1}\sum_{1\leq i\leq n}h(\lambda_i)\xrightarrow{a.s.}\int h(t)dH(t).$$
    
    \item We have from Theorem \ref{thm:opt_shrinkage_der} that for $\abs{x}<2$ (such that $\sigma x$ eventually lies in the support of $H\boxplus \rho_{sc;\sigma^2}$):
    \begin{equation}
        f_h^*(\sigma x)=\int \frac{h(t)dH(t)}{(\sigma^{-1}t-x-\sigma u(\sigma x))^2+\sigma^2v(\sigma x)^2}.
    \end{equation}
    Firstly, we will show that $(x+\sigma u(\sigma x))^2+\sigma^2v(\sigma x)^2\rightarrow 1.$ 
    
    Let $\alpha(z;\sigma)=\sigma m(\sigma z),\Tilde{\alpha}(x)=\lim_{\epsilon\downarrow 0}\alpha(x+i\epsilon).$ Then, we have $$\Tilde{\alpha}(x)=\int \frac{dH(t)}{\sigma^{-1}t-x-\Tilde{\alpha}(x)},$$ so $\Tilde{\alpha(x)}$ converges, as $\sigma\rightarrow \infty,$ to the solution of the equation $\beta=-1/(x+\beta)$ that lies on the upper half plane. Notice that $\beta$ is the limit of the Stieltjes transform of the semicircular distribution with variance $1$ on the real axis. So $$\beta^2+x\beta+1=0\rightarrow (\beta +x)^2-(\beta +x)x+1=0\rightarrow \abs{\beta + x}=1.$$ This is exactly what we claimed, in particular that $$(x+\sigma u(\sigma x))^2+\sigma^2v(\sigma x)^2\xrightarrow {\sigma\rightarrow\infty}1.$$
    
    The rest will follow from Scheff\'e's lemma. In particular, we have for any $x\in\mathbb{R}$ from the equation that defines $m_{\hat{\mu}}:$ $$\int \frac{dH(t)}{(\sigma^{-1}t-x-\sigma u(\sigma x))^2+\sigma^2v(\sigma x)^2}=1.$$ As a consequence, for each $x$ the measure $$\frac{dH(t)}{(\sigma^{-1}t-x-\sigma u(\sigma x))^2+\sigma^2v(\sigma x)^2}$$ is a probability measure that converges weakly to $H$ as $\sigma\downarrow 0.$ The proof is completed. 
    
\end{enumerate}
\end{proof}

\begin{proof}[Proof of Proposition \ref{prop:small_noise}]

\begin{enumerate}
    \item We have \begin{equation}\begin{split}
    \hat{w}_i^\intercal h(A)\hat{w}_i=w_i^\intercal h(A)w_i+2\sigma \frac{d\hat{w}_i}{d\sigma}|_{\sigma=0}h(A)w_i+\omicron{(\sigma)}\\=w_i^\intercal h(A)w_i+2h(\lambda_i)\sigma w_i^\intercal\frac{d\hat{w}_i}{d\sigma}|_{\sigma=0}+\omicron{(\sigma)}.
    \end{split}
    \end{equation}

    Since $\hat{w}_i^\intercal \hat{w}_i=1,$ we get $$w_i^\intercal\frac{d\hat{w}_i}{d\sigma}|_{\sigma=0}=0.$$ We conclude that $$\lim_{\sigma\rightarrow 0}\max_{1\leq i\leq n}\frac{\abs{\hat{w}_i^\intercal h(A)\hat{w}_i-h(\lambda_i)}}{\sigma}=0.$$
    
    \item From the Hadamard variation formulas for the eigenvalues and eigenvectors of $\hat{A}$ (\cite{yau_dynamical}) we know that:
    
    $$\frac{d\hat{\lambda}_i}{d\sigma}|_{\sigma=0}=w_i^\intercal\frac{Z}{\sqrt{n}}w_i$$
    $$\frac{d\hat{w}_i}{d\sigma}|_{\sigma=0}=\sum_{j\neq i}\frac{\frac{w_i^\intercal Zw_j}{\sqrt{n}}}{\lambda_i-\lambda_j}w_j$$
    
    Using these we have for $n$ fixed and $\sigma\rightarrow 0$ we have under the convention $h'(x)=(h(x)-h(y))/(x-y)$ for $x=y:$ \begin{equation}\begin{split}
    &h(\hat{A}) = \sum_{i=1}^nh(\hat{\lambda_i})\hat{w}_i\hat{w}_i^\intercal =\sum_{i=1}^nh\left(\lambda_i+\sigma w_i^\intercal \frac{Z}{\sqrt{n}}w_i\right)\hat{w}_i\hat{w}_i^\intercal+\omicron{(\sigma)}\\
    &=\sum_{i=1}^n\left[h(\lambda_i)+\sigma w_i^\intercal \frac{Z}{\sqrt{n}}w_ih'(\lambda_i)\right]\hat{w}_i\hat{w}_i^\intercal+\omicron{(\sigma)}\\
    &=\sum_{i=1}^n\left[h(\lambda_i)+\sigma \frac{w_i^\intercal Z w_i}{\sqrt{n}}h'(\lambda_i)\right]\left[w_iw_i^\intercal +\sigma \sum_{j\neq i}\frac{\frac{w_i^\intercal Zw_j}{\sqrt{n}}}{\lambda_i-\lambda_j}(w_iw_j^\intercal +w_jw_i^\intercal)\right]+\omicron{(\sigma)}\\
    &=\sum_{i=1}^nh(\lambda_i)w_iw_i^\intercal+\frac{\sigma}{2} \sum_{i,j=1}^n\frac{w_i^\intercal Zw_j}{\sqrt{n}}\frac{h(\lambda_i)-h(\lambda_j)}{\lambda_i-\lambda_j}(w_iw_j^\intercal +w_jw_i^\intercal)+\omicron{(\sigma)}\\
    &=h(A)+\frac{\sigma}{2}\sum_{i,j=1}^n\frac{w_i^\intercal Zw_j}{\sqrt{n}}\frac{h(\lambda_i)-h(\lambda_j)}{\lambda_i-\lambda_j}(w_iw_j^\intercal +w_jw_i^\intercal)+\omicron{(\sigma)}.
    \end{split}
    \end{equation}
    
    This gives us $$\lim_{\sigma \rightarrow 0 }\frac{\Norm{h(\hat{A})-h(A)}_F^2}{n\sigma^2}=\frac{1}{n}\sum_{i,j=1}^n\left( \frac{w_i^\intercal Zw_j}{\sqrt{n}}\right)^2\frac{(h(\lambda_i)-h(\lambda_j))^2}{(\lambda_i-\lambda_j)}.$$ Due to the rotational invariance of $Z$ we can assume that $u_i$ is the $i$-th standard basis vector.
    We get \begin{equation}
        \lim_{\sigma\rightarrow 0}\frac{\Norm{h(\hat{A})-h(A)}_F^2}{n\sigma^2}=n^{-2}\sum_{i,j=1}^nZ_{ij}^2\left(\frac{h(\lambda_i)-h(\lambda_j)}{\lambda_i-\lambda_j}\right)^2.
    \end{equation}
    Writing $$m_{ij}=\left(\frac{h(\lambda_i)-h(\lambda_j)}{\lambda_i-\lambda_j}\right)^2,$$ we know that $m_{ij}\leq \Norm{h'}_{\infty}^2$ and $$n^{-2}\sum_{i,j=1}^nm_{ij}\xrightarrow{a.s.}\iint \frac{(h(t)-h(s))^2}{(t-s)^2}dH(t)dH(s).$$
    
    In addition, $$\mathbb{E}\left[n^{-2}\sum _{i,j=1}^nZ_{ij}^2\frac{(h(\lambda_i)-h(\lambda_j)^2}{(\lambda_i-\lambda_j)^2}|m_{ij},1\leq i,j\leq n\right]=n^{-2}\sum_{i,j=1}^nm_{ij}^2+n^{-2}\sum_{i=1}^nm_{ii}^2,$$
    so \begin{equation}\label{eq:mij_limit}\begin{split}&\mathbb{E}\left[n^{-2}\sum _{i,j=1}^nZ_{ij}^2\frac{(h(\lambda_i)-h(\lambda_j)^2}{(\lambda_i-\lambda_j)^2}|m_{ij},1\leq i,j\leq n\right]\\ 
    &\xrightarrow{a.s.}\iint \frac{(h(t)-h(s))^2}{(t-s)^2}dH(t)dH(s)
    \end{split}
    \end{equation}
    
    Finally, we have $$\Var{n^{-2}\sum_{i,j=1}^nm_{ij}Z_{ij}^2|m_{ij},1\leq{i,j}\leq n}=\mathcal{O}(n^{-2}),$$ so we get $$n^{-2}\sum_{i,j=1}^nm_{ij}Z_{ij}^2-\mathbb{E}\left[n^{-2}\sum _{i,j=1}^nZ_{ij}^2\frac{(h(\lambda_i)-h(\lambda_j)^2}{(\lambda_i-\lambda_j)^2}|m_{ij},1\leq i,j\leq n\right]\xrightarrow{a.s.}0.$$
    
    We deduce from \eqref{eq:mij_limit} that $$n^{-2}\sum_{i,j=1}^nm_{ij}Z_{ij}^2\xrightarrow{a.s.}\iint \frac{(h(t)-h(s))^2}{(t-s)^2}dH(t)dH(s).$$

\end{enumerate}
\end{proof}

\bibliographystyle{plainnat}
\bibliography{final}

\end{document}